\documentclass[final,1p]{elsarticle}

\usepackage{amssymb}
\usepackage{amsthm}

\usepackage{amsfonts,amsmath}

\usepackage[colorlinks, linkcolor=blue, citecolor=blue, pdfborder={0 0 0 [0 0]}]{hyperref}

\numberwithin{equation}{section}

\def\rd{\mathrm{d}}
\def\Lip{\operatorname{Lip}}

\theoremstyle{plain}
\newtheorem{theorem}{Theorem}[section]
\newtheorem{proposition}{Proposition}[section]
\newtheorem{lemma}{Lemma}[section]
\newtheorem{corollary}{Corollary}[section]

\newtheorem*{propositionA}{Proposition A.1}

\theoremstyle{definition}
\newtheorem{definition}{Definition}[section]
\newtheorem{assumption}{Assumption}[section]
\newtheorem{example}{Example}[section]

\theoremstyle{remark}
\newtheorem{remark}{Remark}[section]

\journal{Journal of Functional Analysis}

\begin{document}

\begin{frontmatter}

\title{Equivalence of minimax and viscosity solutions of path-dependent Hamilton--Jacobi equations}

\cortext[cor1]{Corresponding author}

\author[aff1,aff2]{M.I. Gomoyunov\corref{cor1}}
\ead{m.i.gomoyunov@gmail.com}
\author[aff1,aff2]{A.R. Plaksin}
\ead{a.r.plaksin@gmail.com}

\affiliation[aff1]{organization={N.N. Krasovskii Institute of Mathematics and Mechanics of the Ural Branch of the Russian
            Academy of Sciences},
            addressline={16 S. Kovalevskaya Str.},
            city={Ekaterinburg},
            postcode={620108},
            country={Russia}}

\affiliation[aff2]{organization={Ural Federal University},
            addressline={19 Mira Str.},
            city={Ekaterinburg},
            postcode={620002},
            country={Russia}}

\begin{abstract}
    In the paper, we consider a path-dependent Hamilton--Jacobi equation with coinvariant derivatives over the space of continuous functions.
    Such equations arise from optimal control problems and differential games for time-delay systems.
    We study generalized solutions of the considered Hamilton--Jacobi equation both in the minimax and in the viscosity sense.
    A minimax solution is defined as a functional which epigraph and subgraph satisfy certain conditions of weak invariance, while a viscosity solution is defined in terms of a pair of inequalities for coinvariant sub- and supergradients.
    We prove that these two notions are equivalent, which is the main result of the paper.
    As a corollary, we obtain comparison and uniqueness results for viscosity solutions of a Cauchy problem for the considered Hamilton--Jacobi equation and a right-end boundary condition.
    The proof is based on a certain property of the coinvariant subdifferential.
    To establish this property, we develop a technique going back to the proofs of multidirectional mean-value inequalities.
    In particular, the absence of the local compactness property of the underlying continuous function space is overcome by using Borwein--Preiss variational principle with an appropriate guage-type functional.
\end{abstract}


\end{frontmatter}

\section{Introduction}

    In the theory of Hamilton--Jacobi (HJ) equations (nonlinear first-order partial differential equations), several approaches to a notion of a generalized solution are known.
    Among others, the minimax and viscosity approaches are the most developed.

    The minimax approach originates in the positional differential games theory (see, e.g., \cite{Krasovskii_Subbotin_1988,Krasovskii_Krasovskii_1995}) and can be seen as the further development of the unification constructions of differential games \cite{Krasovskii_1976}.
    According to this approach \cite{Subbotin_1984,Subbotin_1995}, a generalized (minimax) solution is defined as a function satisfying a pair of non-local conditions expressing properties of weak invariance of the epigraph and subgraph of the function with respect to so-called characteristic differential inclusions.
    In infinitesimal form, these conditions reduce to inequalities for lower and upper Dini directional derivatives of the function.
    On the other hand, according to the viscosity approach \cite{Crandall_Lions_1983,Crandall_Evans_Lions_1984}, which goes back to the vanishing viscosity method in mathematical physics, a generalized (viscosity) solution is defined as a function satisfying a pair of conditions involving smooth test functions.
    In turn, these conditions are equivalent to inequalities for sub- and supergradients of the function.
    Within both approaches, various properties of generalized (minimax and viscosity) solutions were obtained.
    In particular, these results were shown to be a useful tool in the study of optimal control problems and differential games.

    Despite the fact that the notions of minimax and viscosity solutions have different origins, it was proved that they are equivalent under quite general assumptions.
    This result was established in \cite{Lions_Souganidis_1985,Subbotin_Tarasyev_1986} for the case of Bellman--Isaacs equations and under an additional local Lipschitz continuity condition.
    In a more general setting, a direct proof of this result was given in \cite{Subbotin_1993_MSb} (see also \cite{Subbotin_1991_NA}) on the basis of the following Property of Subdifferential:
    if a lower Dini directional derivative of a function at some point is positive in a convex cone of directions, then, in an arbitrarily small neighbourhood of this point, there exists a point at which the function has a subgradient belonging to the cone dual to the original cone of directions.
    Later, it was observed that this property could be obtained as an infinitesimal version of Multidirectional Mean-Value Inequality \cite{Clarke_Ledyaev_1993_DAN} (in this connection, see also \cite[Section A6]{Subbotin_1995}).

    Note that, in \cite{Clarke_Ledyaev_1994} (see also \cite[Theorem 2.6]{Clarke_Ledyaev_Stern_Wolenski_1998}), Multidirectional Mean-Value Inequality was proved for functionals on a Hilbert space.
    Moreover, an analog of Property of Sub\-dif\-ferential was derived and used, in particular, to establish the equivalence of the minimax and viscosity solutions of HJ equations in Hilbert spaces.
    For further generalizations and more recent developments in this direction, the reader is referred to \cite{Kipka_Ledyaev_2019} and the references therein.
    One of the difficulties that arise in the case of an infinite-dimensional space in comparison with the finite-dimensional setting is related to the absence of the local compactness property.
    Indeed, a lower semi-continuous functional that is bounded from below may not attain its minimum over a closed bounded set.
    This is overcome with the help of smooth variational principles, which assert that a small smooth perturbation can be added to the functional in question so that the perturbed functional attains its minimum.
    In particular, Borwein--Preiss variational principle \cite{Borwein_Preiss_1987} (see also, e.g., \cite[Theorems 2.5.2 and 2.5.3]{Borwein_Zhu_2005}) was used in \cite{Clarke_Ledyaev_1994}, Stegall variational principle \cite{Stegall_1978} (see also, e.g., \cite[Theorem 6.3.5]{Borwein_Zhu_2005}) was used in \cite[Theorem 2.6]{Clarke_Ledyaev_Stern_Wolenski_1998}, and Deville--Godefroy--Zizler variational principle \cite{Deville_Godefroy_Zizler_1993} (see also, e.g., \cite[Theorems 2.5.4 and 2.5.7]{Borwein_Zhu_2005}) was used in \cite{Kipka_Ledyaev_2019}.
    However, a drawback of such smooth variational principles is that they require certain smoothness properties of the norm of the underlying space, which, in particular, does not allow one to directly apply them for functionals on the continuous function space equipped with the uniform norm.
    This circumstance also complicates the development of the theory of infinite-dimensional HJ equations \cite{Crandall_Lions_1985,Crandall_Lions_1986} considered over this space.

    The goal of the preset paper is to prove the equivalence of the minimax and viscosity approaches to a notion of a generalized solution of so-called path-dependent HJ equations, which arise naturally from optimal control problems and differential games for time-delay systems (see, e.g., \cite{Kim_1999,Lukoyanov_2000_JAMM,Lukoyanov_2003_2,Lukoyanov_2010_IMM_Eng_1,Pham_Zhang_2014,Tang_Zhang_2015,Kaise_2015,Bayraktar_Keller_2018,Saporito_2019} and also \cite{Lukoyanov_2011_Eng}).
    Such equations are usually considered over the space of continuous functions, and, instead of the classical Fr\'{e}chet derivatives as in, e.g., \cite{Crandall_Lions_1985,Crandall_Lions_1986}, they involve some special pathwise derivatives of functionals.
    In this paper, we follow \cite{Kim_1999,Lukoyanov_2000_JAMM} (see also \cite{Lukoyanov_2011_Eng}) and use so-called coinvariant ($ci$-) derivatives, while one can also use so-called horizontal and vertical derivatives \cite{Dupire_2009}.
    For a discussion of the relationship between these notions, the reader is referred to \cite[Section 5.2]{Gomoyunov_Lukoyanov_Plaksin_2021}.
    Minimax solutions of path-dependent HJ equations were introduced in \cite{Lukoyanov_2000_JAMM,Lukoyanov_2003_1} and comprehensively studied in, e.g., \cite{Lukoyanov_2010_IMM_Eng_1,Lukoyanov_2011_Eng,Bayraktar_Keller_2018,Gomoyunov_Lukoyanov_Plaksin_2021}.
    On the other hand, in the development of the viscosity approach, the absence of the local compactness property of the continuous function space becomes a more essential problem.
    To overcome this, based on the idea of \cite{Soner_1988}, it was proposed in \cite{Lukoyanov_2007_IMM_Eng} to consider a modified notion of a viscosity solution.
    This modification consists in introducing into the definition a special expanding sequence of compact sets which union is everywhere dense.
    Note also that the role of test functionals in this definition is played by functionals that are smooth in the sense of $ci$-differentiability (rather than Fr\'{e}chet differentiability).
    This technique and some of its variants and generalizations, which also allow one to use compactness arguments in the proofs, were further developed in, e.g., \cite{Lukoyanov_2010_IMM_Eng_1,Kaise_2015,Tang_Zhang_2015,Ekren_Touzi_Zhang_2016_1,Ekren_Touzi_Zhang_2016_2,Zhou_2020_1}.

    However, recently, several papers have appeared \cite{Zhou_2015,Zhou_2018_COCV,Plaksin_2020_JOTA,Zhou_2020_2,Cosso_Gozzi_Rosestolato_Russo_2021,Plaksin_2021_SIAM,Zhou_2021,Zhou_2021_2,Cosso_Russo_2022} in which it is proposed to consider another notions of viscosity solutions of path-dependent HJ equations that are more consistent with the classical case.
    The present paper is also motivated by the same reasons and can be assigned to this area of research.
    Let us shortly describe some of the basic ideas of the cited papers that, in particular, allowed one to overcome the problem with the local compactness property.

    \smallskip

    \noindent (i)
    In \cite{Plaksin_2020_JOTA,Plaksin_2021_SIAM}, it was suggested to move from the space of continuous functions to the wider space of piecewise continuous functions and to deal with a certain class of locally Lipschitz continuous solutions only.
    In application to optimal control problems in time-delay systems, this means the separation of the current value of the state vector and the history, while the Lipschitz condition is expressed in terms of the Euclidean norm of the current state and the integral norm of the history.
    In this case, it was shown that the proof of the uniqueness result for viscosity solutions can be carried out relying on the finite-dimensional optimization technique only.

    \smallskip

    \noindent (ii)
    Similar constructions were developed in \cite{Zhou_2015,Zhou_2018_COCV,Zhou_2021}.
    The key tool used in the proof of the uniqueness result is a so-called left maximization principle in the space of piecewise continuous functions, which is established for functionals satisfying certain growth estimate and uniform continuity property (close to that from item (i)).
    The proof of this principle also uses only finite-dimensional optimization arguments.

    \smallskip

    \noindent (iii)
    A different approach was proposed in \cite{Zhou_2020_2,Cosso_Gozzi_Rosestolato_Russo_2021,Zhou_2021_2,Cosso_Russo_2022}, where features of optimization problems for functionals over the space of continuous functions were taken into account in a direct way.
    The progress is related to the construction of some special smooth (in the sense of $ci$-differentiability or horizontal and vertical differentiability) functionals, which allows one to built smooth gauge-type functionals and apply Borwein--Preiss variational principle from \cite[Theorem 1]{Li_Shi_2000} (see also, e.g., \cite[Theorem 2.5.2]{Borwein_Zhu_2005}).
    In particular, in \cite{Zhou_2020_1}, a smooth functional was found that in some sense equivalent to the square of the uniform norm, and this functional formed the basis of the results of \cite{Zhou_2020_2,Zhou_2021_2} (see also \cite{Gomoyunov_Lukoyanov_Plaksin_2021}).

    \smallskip

    In this paper, we deal with viscosity solutions of the considered path-dependent HJ equation that are defined in terms of a pair of inequalities for so-called coinvariant ($ci$-) sub- and supergradients, which are naturally consistent with the notion of $ci$-differen\-tiability.
    This definition was given in \cite[Section 14]{Lukoyanov_2011_Eng}, but no uniqueness results were established for such a viscosity solution.
    Our main result is that, under certain assumptions, this definition of a viscosity solution is equivalent to that of a minimax solution from, e.g., \cite{Gomoyunov_Lukoyanov_Plaksin_2021}.
    In particular, as a direct corollary, we obtain comparison and uniqueness theorems for viscosity solutions of a Cauchy problem for the considered Hamilton--Jacobi equation and a right-end boundary condition.
    In general, our proof follows the scheme from \cite[Theorem 4.3]{Subbotin_1995}, which was already partially adapted to the path-dependent setting in \cite{Plaksin_2020_JOTA}.
    Namely, we establish a property of the $ci$-subdifferential, which is similar to Property of Subdifferential \cite{Subbotin_1993_MSb}.
    To this end, we apply the technique going back to the use of Multidirectional Mean-Value Inequality \cite{Clarke_Ledyaev_1994}.
    In the spirit of \cite{Cosso_Russo_2022}, the absence of the local compactness is overcome by constructing a guage-type functional, which is carried out on the basis of the functional from \cite{Zhou_2020_1}, and by applying Borwein--Preiss variational principle from \cite[Theorem 1]{Li_Shi_2000}.
    It should be emphasized that the result is proved under an additional assumption concerning continuity properties of the considered solutions, which are close to the Lipschitz condition described above in item (i).

    Note also that the comparison and uniqueness theorems for viscosity solutions of Cauchy problems for path-dependent HJ equations obtained in the present paper are not covered by the results of \cite{Zhou_2015,Plaksin_2020_JOTA,Zhou_2020_2,Cosso_Gozzi_Rosestolato_Russo_2021,Plaksin_2021_SIAM,Zhou_2021}, since, on the one hand, we do not require any growth estimates as well as any uniform continuity or boundedness properties of the solutions we are dealing with, and, on the other hand, we assume a weaker set of conditions on the Hamiltonian and the boundary functional.
    A more detailed discussion of the presented results is given in the main body of the paper.

    The paper is organized as follows.
    In Section \ref{section_Formulation_of_main_result}, we describe the path-dependent HJ equation under consideration, introduce the notions of minimax and viscosity solutions of this equation, and formulate our main result about their equivalence.
    Section \ref{section_Discussion} is devoted to a discussion of this result and some of its corollaries.
    In particular, we derive comparison and uniqueness theorems for viscosity solutions, obtain several criteria for viscosity solutions, and present an application to optimal control problems in time-delay systems, which also implies an existence result for a viscosity solution in the case of path-dependent Bellman equations.
    In Section \ref{section_proof}, we give the proof of the main result.

\section{Formulation of main result}
\label{section_Formulation_of_main_result}

    \subsection{Basic notation}

        Let $n \in \mathbb{N}$, $h > 0$, and $T > 0$ be fixed.
        Let $\mathbb{R}^n$ be the $n$-dimensional Euclidean space with the inner product $\langle \cdot, \cdot \rangle$ and the norm $\|\cdot\|$.
        For every $t \in [0, T]$, let $C([- h, t], \mathbb{R}^n)$ be the Banach space of all continuous functions $x \colon [- h, t] \to \mathbb{R}^n$ equipped with the norm
        \begin{equation} \label{norm}
            \|x(\cdot)\|_{[- h, t]}
            \triangleq \max_{\xi \in [- h, t]} \|x(\xi)\|,
            \quad \forall x(\cdot) \in C([- h, t], \mathbb{R}^n).
        \end{equation}

        Let us consider the set
        \begin{equation*}
            G
            \triangleq \bigcup_{t \in [0, T]} \bigl( \{t\} \times C([- h, t], \mathbb{R}^n) \bigr).
        \end{equation*}
        In other words, $G$ consists of all pairs $(t, x(\cdot))$ with $t \in [0, T]$ and $x(\cdot) \in C([- h, t], \mathbb{R}^n)$.
        The set $G$ is equipped with the metric $\rho_\infty \colon G \times G \to \mathbb{R}$ given by (see, e.g., \cite[Section 5.1]{Gomoyunov_Lukoyanov_Plaksin_2021} and the references therein)
        \begin{equation} \label{rho_infty}
            \rho_\infty \bigl( (t, x(\cdot)), (\tau, y(\cdot)) \bigr)
            \triangleq |t - \tau| + \|x(\cdot \wedge t) - y(\cdot \wedge \tau)\|_{[- h, T]},
            \ \ \forall (t, x(\cdot)), (\tau, y(\cdot)) \in G.
        \end{equation}
        Here and below, we use the notation $a \wedge b \triangleq \min\{a, b\}$ for all $a$, $b \in \mathbb{R}$, and, respectively, for every $(t, x(\cdot)) \in G$, we define the function $x(\cdot \wedge t) \in C([- h, T], \mathbb{R}^n)$ as follows:
        \begin{equation*}
            x(\xi \wedge t)
            \triangleq \begin{cases}
                x(\xi), & \mbox{if } \xi \in [- h, t], \\
                x(t), & \mbox{if } \xi \in (t, T].
            \end{cases}
        \end{equation*}
        The metric space $(G, \rho_\infty)$ is complete due to completeness of the space $C([- h, T], \mathbb{R}^n)$.
        We also introduce the sets
        \begin{equation*}
            G_0
            \triangleq \bigl\{ (t, x(\cdot)) \in G \bigm| t < T \bigr\}
        \end{equation*}
        and
        \begin{equation} \label{G_alpha}
            G(\alpha)
            \triangleq \bigl\{( t, x(\cdot)) \in G \bigm|
            \|x(\cdot)\|_{[- h, t]}
            \leq \alpha \bigr\},
            \quad \forall \alpha > 0.
        \end{equation}

        Along with the metric $\rho_\infty$, we consider the auxiliary metric $\rho_1 \colon G \times G \to \mathbb{R}$ given by
        \begin{equation} \label{rho_1}
            \rho_1 \bigl( (t, x(\cdot)), (\tau, y(\cdot)) \bigr)
            \triangleq |t - \tau| + \|x(t) - y(\tau)\|
            + \int_{- h}^{T} \|x(\xi \wedge t) - y(\xi \wedge \tau)\| \, \rd \xi
        \end{equation}
        for all $(t, x(\cdot))$, $(\tau, y(\cdot)) \in G$.
        Note that
        \begin{equation} \label{rho_1_and_rho_infty}
            \rho_1 \bigl( (t, x(\cdot)), (\tau, y(\cdot)) \bigr)
            \leq (1 + T + h) \rho_\infty \bigl( (t, x(\cdot)), (\tau, y(\cdot)) \bigr),
            \quad \forall (t, x(\cdot)), (\tau, y(\cdot)) \in G.
        \end{equation}

        For a function $z(\cdot) \in C([- h, T], \mathbb{R}^n)$ and a point $t \in [0, T]$, let $z_t(\cdot) \in C([- h, t], \mathbb{R}^n)$  denote the restriction of the function $z(\cdot)$ to the interval $[- h, t]$, i.e.,
        \begin{equation} \label{x_t}
            z_t(\xi)
            \triangleq z(\xi),
            \quad \forall \xi \in [- h, t].
        \end{equation}

        For every $(t, x(\cdot)) \in G$, let $\Lip(t, x(\cdot))$ be the set of all functions $z(\cdot) \in C([- h, T], \mathbb{R}^n)$ that satisfy the condition $z_t(\cdot) = x(\cdot)$ and are Lipschitz continuous on $[t, T]$.

    \subsection{Coinvariant derivatives}

        Let us recall the notion of coinvariant ($ci$-) differentiability (see, e.g., \cite{Gomoyunov_Lukoyanov_Plaksin_2021} and the references therein).

        A functional $\varphi \colon G \to \mathbb{R}$ is called {\it $ci$-differentiable} at a point $(t, x(\cdot)) \in G_0$ if there exist $\partial_t \varphi(t, x(\cdot)) \in \mathbb{R}$ and $\nabla \varphi(t, x(\cdot)) \in \mathbb{R}^n$ such that, for all $z(\cdot) \in \Lip(t, x(\cdot))$,
        \begin{equation} \label{ci_derivatives}
            \lim_{\tau \downarrow t}
            \frac{\varphi(\tau, z_\tau(\cdot)) - \varphi(t, x(\cdot))
            - \partial_t \varphi(t, x(\cdot)) (\tau - t) - \langle \nabla \varphi(t, x(\cdot)), z(\tau) - x(t) \rangle}{\tau - t}
            = 0.
        \end{equation}
        In this case, the values $\partial_t \varphi(t, x(\cdot))$ and $\nabla \varphi(t, x(\cdot))$ are called {\it $ci$-derivatives} of $\varphi$ at $(t, x(\cdot))$.
        Note that, relation \eqref{ci_derivatives} can be rewritten as follows:
        \begin{equation*}
            \varphi(\tau, z_\tau(\cdot)) - \varphi(t, x(\cdot))
            = \partial_t \varphi(t, x(\cdot)) (\tau - t) + \langle \nabla \varphi(t, x(\cdot)), z(\tau) - x(t) \rangle + o(\tau - t)
        \end{equation*}
        for all $\tau \in (t, T]$, where the function $o \colon (0, \infty) \to \mathbb{R}$, which may depend on $t$ and $z(\cdot)$, satisfies the condition $o(\delta) / \delta \to 0$ as $\delta \downarrow 0$.

        A functional $\varphi \colon G \to \mathbb{R}$ is called {\it $ci$-smooth} if it is continuous, $ci$-differ\-en\-tiable at every point $(t, x(\cdot)) \in G_0$, and the mappings $\partial_t \varphi \colon G_0 \to \mathbb{R}$ and $\nabla \varphi \colon G_0 \to \mathbb{R}^n$ are continuous.

    \subsection{Path-dependent Hamilton--Jacobi equation}

        This paper deals with the following {\it path-dependent Hamilton--Jacobi (HJ) equation} with $ci$-derivatives
        \begin{equation} \label{HJ}
            \partial_t \varphi(t, x(\cdot)) + H\bigl( t, x(\cdot), \nabla \varphi(t, x(\cdot)) \bigr)
            = 0,
            \quad \forall (t, x(\cdot)) \in G_0.
        \end{equation}
        In this equation, the {\it Hamiltonian} $G \times \mathbb{R}^n \ni ((t, x(\cdot)), s) \mapsto H(t, x(\cdot), s) \in \mathbb{R}$ is given, and the unknown is a functional $\varphi \colon G \to \mathbb{R}$.

        \begin{assumption} \label{assumption_H1_H2}
            The Hamiltonian $H$ satisfies the conditions below:

            \smallskip

            \noindent {\rm(i)}
                For every $s \in \mathbb{R}^n$, the functional $G \ni (t, x(\cdot)) \mapsto H(t, x(\cdot), s) \in \mathbb{R}$ is continuous.

            \smallskip

            \noindent {\rm (ii)}
                There exists $c_H > 0$ such that, for all $(t, x(\cdot)) \in G$ and all $s$, $r \in \mathbb{R}^n$,
                \begin{equation*}
                    | H(t, x(\cdot), s) - H(t, x(\cdot), r) |
                    \leq c_H (1 + \|x(\cdot)\|_{[- h, t]}) \|s - r\|.
                \end{equation*}
        \end{assumption}

        Note that Assumption \ref{assumption_H1_H2} implies continuity of the Hamiltonian $H$.

        In the literature, for the HJ equation \eqref{HJ}, a {\it Cauchy problem} is usually studied under the right-end {\it boundary condition}
        \begin{equation} \label{boundary_condition}
            \varphi(T, x(\cdot))
            = \sigma(x(\cdot)),
            \quad \forall x(\cdot) \in C([- h, T], \mathbb{R}^n),
        \end{equation}
        where $\sigma \colon C([- h, T], \mathbb{R}^n) \to \mathbb{R}$ is a given {\it boundary functional}, which is assumed to be at least continuous.
        However, since the presence of a specific boundary condition is not essential for the main result of the paper, we will mainly focus on minimax and viscosity solutions of the HJ equation \eqref{HJ} rather than those of the Cauchy problem \eqref{HJ}, \eqref{boundary_condition}.

    \subsection{Minimax solutions}
    \label{subsection_minimax}

        In order to give a definition of a minimax solution of the HJ equation \eqref{HJ}, we need to introduce some notation.

        Let us denote by $\Phi_+$ the set of all functionals $\varphi \colon G \to \mathbb{R}$ that satisfy the following {\it condition of lower semi-continuity}:
        for every point $(t, x(\cdot)) \in G$ and every sequence $\{(t^{(k)}, x^{(k)}(\cdot))\}_{k = 1}^\infty \subset G$, the inequality
        \begin{equation} \label{Phi_+}
            \liminf_{k \to \infty} \varphi(t^{(k)}, x^{(k)}(\cdot))
            \geq \varphi(t, x(\cdot))
        \end{equation}
        is valid provided that $\rho_1 ((t^{(k)}, x^{(k)}(\cdot)), (t, x(\cdot))) \to 0$ as $k \to \infty$ and there exists $\alpha > 0$ such that $\{(t^{(k)}, x^{(k)}(\cdot))\}_{k = 1}^\infty \subset G(\alpha)$.
        Note that, if a functional $\varphi \colon G \to \mathbb{R}$ is lower semi-continuous with respect to the auxiliary metric $\rho_1$, then $\varphi \in \Phi_+$.
        On the other hand, any functional $\varphi \in \Phi_+$ is lower semi-continuous (with respect to the basic metric $\rho_\infty$).

        Analogously, let $\Phi_-$ be the set of all functionals $\varphi \colon G \to \mathbb{R}$ that satisfy the following {\it upper semi-continuity condition}:
        for every $(t, x(\cdot)) \in G$ and every $\{(t^{(k)}, x^{(k)}(\cdot))\}_{k = 1}^\infty \subset G$, the inequality
        \begin{equation} \label{Phi_-}
            \limsup_{k \to \infty} \varphi(t^{(k)}, x^{(k)}(\cdot))
            \leq \varphi(t, x(\cdot))
        \end{equation}
        is fulfilled if $\rho_1 ((t^{(k)}, x^{(k)}(\cdot)), (t, x(\cdot))) \to 0$ as $k \to \infty$ and $\{(t^{(k)}, x^{(k)}(\cdot))\}_{k = 1}^\infty \subset G(\alpha)$ for some $\alpha > 0$.

        Further, for every $(t, x(\cdot)) \in G$, let us consider the sets
        \begin{equation} \label{B}
            B_{c_H}(t, x(\cdot))
            \triangleq \bigl\{ l \in \mathbb{R}^n \bigm|
            \|l\| \leq c_H (1 + \|x(\cdot)\|_{[- h, t]}) \bigr\},
        \end{equation}
        where $c_H$ is the constant from condition (ii) of Assumption \ref{assumption_H1_H2}, and
        \begin{equation*}
            \Lip_{c_H}(t, x(\cdot))
            \triangleq \bigl\{ z(\cdot) \in \Lip(t, x(\cdot)) \bigm|
            \dot{z}(\tau) \in B_{c_H}(\tau, z_\tau(\cdot)) \text{ for a.e. } \tau \in [t, T] \bigr\},
        \end{equation*}
        where $\dot{z}(\tau) \triangleq \rd z(\tau) / \rd \tau$.

        Finally, for $\varphi \colon G \to \mathbb{R}$ and $s \in \mathbb{R}^n$, let $\varphi_s \colon G \to \mathbb{R}$ denote the functional defined by
        \begin{equation} \label{varphi_s}
            \varphi_s (t, x(\cdot))
            \triangleq \varphi(t, x(\cdot)) - \langle s, x(t) \rangle,
            \quad \forall (t, x(\cdot)) \in G.
        \end{equation}

        Following, e.g., \cite{Gomoyunov_Lukoyanov_Plaksin_2021} (see also the references therein), we give
        \begin{definition} \label{definition_minimax}
            Let $\varphi \colon G \to \mathbb{R}$.

            \smallskip

            \noindent (i)
                We call $\varphi$ an {\it upper minimax solution} of the HJ equation \eqref{HJ} if $\varphi \in \Phi_+$ and
                \begin{equation} \label{upper_minimax}
                    \begin{array}{c}
                        \displaystyle
                        \inf_{z(\cdot) \in \Lip_{c_H}(t, x(\cdot))}
                        \biggl( \varphi_s(\tau, z_\tau(\cdot))
                        + \int_t^\tau H(\xi, z_\xi(\cdot), s) \, \rd \xi \biggr)
                        \leq \varphi_s(t, x(\cdot)), \\[1.2em]
                        \forall (t, x(\cdot)) \in G_0, \quad \forall \tau \in (t, T], \quad \forall s \in \mathbb{R}^n.
                    \end{array}
                \end{equation}

            \smallskip

            \noindent (ii)
                We call $\varphi$ a {\it lower minimax solution} of the HJ equation \eqref{HJ} if $\varphi \in \Phi_-$ and
                \begin{equation} \label{lower_minimax}
                    \begin{array}{c}
                        \displaystyle
                        \sup_{z(\cdot) \in \Lip_{c_H}(t, x(\cdot))}
                        \biggl( \varphi_s(\tau, z_\tau(\cdot))
                        + \int_t^\tau H(\xi, z_\xi(\cdot), s) \, \rd \xi \biggr)
                        \geq \varphi_s(t, x(\cdot)), \\[1.2em]
                        \forall (t, x(\cdot)) \in G_0, \quad \forall \tau \in (t, T], \quad \forall s \in \mathbb{R}^n.
                    \end{array}
                \end{equation}

            \smallskip

            \noindent (iii)
                We call $\varphi$ a {\it minimax solution} of the HJ equation \eqref{HJ} if $\varphi$ is an upper as well as a lower minimax solution of this equation.
        \end{definition}

    \subsection{Viscosity solutions}

        Let us recall the notions of coinvariant ($ci$-) sub- and superdifferentials introduced in \cite[Section 14]{Lukoyanov_2011_Eng}.

        Let $\varphi \colon G \to \mathbb{R}$ and $(t, x(\cdot)) \in G_0$.
        For every $z(\cdot) \in \Lip(t, x(\cdot))$, let us consider the {\it lower} and {\it upper right derivatives} of the functional $\varphi$ at the point $(t, x(\cdot))$ {\it along the function} $z(\cdot)$ given by
        \begin{equation} \label{derivatives_along_extenstion}
            \partial^- \varphi(t, x(\cdot); z(\cdot))
            \triangleq \liminf_{\tau \downarrow t} \frac{\varphi(\tau, z_\tau(\cdot)) - \varphi(t, x(\cdot))}{\tau - t}
        \end{equation}
        and
        \begin{equation*}
            \partial^+ \varphi(t, x(\cdot); z(\cdot))
            \triangleq \limsup_{\tau \downarrow t} \frac{\varphi(\tau, z_\tau(\cdot)) - \varphi(t, x(\cdot))}{\tau - t},
        \end{equation*}
        respectively.
        Then, the {\it $ci$-subdifferential} and {\it $ci$-superdifferential} of $\varphi$ at $(t, x(\cdot))$ are introduced as follows:
        \begin{equation} \label{subdifferential_via_directional}
            D^- \varphi(t, x(\cdot))
            \triangleq \bigl\{ (p_0, p) \in \mathbb{R} \times \mathbb{R}^n \bigm|
            p_0
            \leq \partial^- \varphi_p (t, x(\cdot); z(\cdot)),
            \, \forall z(\cdot) \in \Lip(t, x(\cdot)) \bigr\}
        \end{equation}
        and
        \begin{equation} \label{superdifferential_via_directional}
            D^+ \varphi(t, x(\cdot))
            \triangleq \bigl\{ (q_0, q) \in \mathbb{R} \times \mathbb{R}^n \bigm|
            q_0
            \geq \partial^+ \varphi_q (t, x(\cdot); z(\cdot)),
            \, \forall z(\cdot) \in \Lip(t, x(\cdot)) \bigr\},
        \end{equation}
        respectively, where the functionals $\varphi_p$, $\varphi_q$ are defined by $\varphi$ and $p$, $q$ according to \eqref{varphi_s}.
        Rewriting \eqref{subdifferential_via_directional} and \eqref{superdifferential_via_directional} in explicit form, we obtain (compare with the definition of $ci$-differentiability, see \eqref{ci_derivatives})
        \begin{multline} \label{subdifferential}
            D^- \varphi(t, x(\cdot))
            = \biggl\{ (p_0, p) \in \mathbb{R} \times \mathbb{R}^n \biggm|
            \forall z(\cdot) \in \Lip(t, x(\cdot)) \\
            \liminf_{\tau \downarrow t}
            \frac{\varphi(\tau, z_\tau(\cdot)) - \varphi(t, x(\cdot)) - p_0 (\tau - t) - \langle p, z(\tau) - x(t) \rangle}{\tau - t}
            \geq 0 \biggr\}
        \end{multline}
        and
        \begin{multline*}
            D^+ \varphi(t, x(\cdot))
            = \biggl\{ (q_0, q) \in \mathbb{R} \times \mathbb{R}^n \biggm|
            \forall z(\cdot) \in \Lip(t, x(\cdot)) \\
            \limsup_{\tau \downarrow t}
            \frac{\varphi(\tau, z_\tau(\cdot)) - \varphi(t, x(\cdot)) - q_0 (\tau - t) - \langle q, z(\tau) - x(t) \rangle}{\tau - t}
            \leq 0 \biggr\}.
        \end{multline*}
        Note also that, if $\varphi$ is $ci$-differentiable at $(t, x(\cdot))$, then
        \begin{equation*}
            D^- \varphi (t, x(\cdot))
            = \bigl\{ (p_0, p) \in \mathbb{R} \times \mathbb{R}^n \bigm|
            p_0
            \leq \partial_t \varphi(t, x(\cdot)),
            \, p
            = \nabla \varphi(t, x(\cdot)) \bigr\}
        \end{equation*}
        and
        \begin{equation*}
            D^+ \varphi(t, x(\cdot))
            = \bigl\{ (q_0, q) \in \mathbb{R} \times \mathbb{R}^n \bigm|
            q_0
            \geq \partial_t \varphi(t, x(\cdot)),
            \, q
            = \nabla \varphi(t, x(\cdot)) \bigr\}.
        \end{equation*}

        In accordance with \cite[Section 14]{Lukoyanov_2011_Eng}, we give
        \begin{definition} \label{definition_viscosity}
            Let $\varphi \colon G \to \mathbb{R}$.

            \smallskip

            \noindent (i)
                We call $\varphi$ an {\it upper viscosity solution} of the HJ equation \eqref{HJ} if $\varphi \in \Phi_+$ and
                \begin{equation} \label{upper_viscosity}
                    p_0 + H(t, x(\cdot), p)
                    \leq 0,
                    \quad \forall (p_0, p) \in D^- \varphi(t, x(\cdot)), \quad \forall (t, x(\cdot)) \in G_0.
                \end{equation}

            \smallskip

            \noindent (ii)
                We call $\varphi$ a {\it lower viscosity solution} of the HJ equation \eqref{HJ} if $\varphi \in \Phi_-$ and
                \begin{equation} \label{lower_viscosity}
                    q_0 + H(t, x(\cdot), q)
                    \geq 0,
                    \quad \forall (q_0, q) \in D^+ \varphi (t, x(\cdot)), \quad \forall (t, x(\cdot)) \in G_0.
                \end{equation}

            \smallskip

            \noindent (iii)
                We call $\varphi$ a {\it viscosity solution} of the HJ equation \eqref{HJ} if it is an upper as well as a lower viscosity solution of this equation.
        \end{definition}

        Note that Definition \ref{definition_viscosity} can be considered as a quite natural extension of the definition of viscosity solutions of HJ equations with partial derivatives in terms of inequalities for sub- and supergradients from, e.g., \cite{Crandall_Evans_Lions_1984} (see also, e.g., \cite[Sections 4.4 and 6.3]{Subbotin_1995}) to the path-dependent setting.

    \subsection{Main result}

        The main result of the paper is
        \begin{theorem} \label{theorem_main}
            Let Assumption \ref{assumption_H1_H2} hold.
            Then, a functional $\varphi \colon G \to \mathbb{R}$ is an upper minimax (resp., a lower minimax, resp. a minimax) solution of the HJ equation \eqref{HJ} if and only if $\varphi$ is an upper viscosity (resp., a lower viscosity, resp. a viscosity) solution of this equation.
        \end{theorem}

        In other words, Theorem \ref{theorem_main} states that, under Assumption \ref{assumption_H1_H2}, conditions \eqref{upper_minimax} and \eqref{upper_viscosity} are equivalent for every functional $\varphi \in \Phi_+$, while conditions \eqref{lower_minimax} and \eqref{lower_viscosity} are equivalent for every functional $\varphi \in \Phi_-$.

        The proof of Theorem \ref{theorem_main} is given in Section \ref{section_proof}.
        It is carried out by the scheme from \cite[Theorem 4.3]{Subbotin_1995}, where a similar result is obtained for HJ equations with partial derivatives (more precisely, see the proof of the equivalence $(U5) \Leftrightarrow (U6)$ in \cite[Section 4.6]{Subbotin_1995}).
        Note that this scheme was already partially adapted to the path-dependent setting in \cite{Plaksin_2020_JOTA}.

        The next section is devoted to a discussion of Theorem \ref{theorem_main} and some of its corollaries.

\section{Discussion of main result and its corollaries}
\label{section_Discussion}

    It is convenient to divide the discussion into three parts.

    \subsection{Comparison theorem for viscosity solutions}

        This section deals with the Cauchy problem for the HJ equation \eqref{HJ} under the boundary condition \eqref{boundary_condition}.
        Let us give definitions of minimax and viscosity solutions of this problem.

        \begin{definition}
            Let $\varphi \colon G \to \mathbb{R}$.

            \smallskip

            \noindent (i)
                We call $\varphi$ an {\it upper minimax} (resp., an {\it upper viscosity}) solution of the Cauchy problem \eqref{HJ}, \eqref{boundary_condition} if $\varphi$ is an upper minimax (resp., an upper viscosity) solution of the HJ equation \eqref{HJ} and
                \begin{equation*}
                    \varphi(T, x(\cdot))
                    \geq \sigma(x(\cdot)),
                    \quad \forall x(\cdot) \in C([- h, T], \mathbb{R}^n).
                \end{equation*}

            \smallskip

            \noindent (ii)
                We call $\varphi$ a {\it lower minimax} (resp., a {\it lower viscosity}) solution of the Cauchy problem \eqref{HJ}, \eqref{boundary_condition} if $\varphi$ is a lower minimax (resp., a lower viscosity) solution of the HJ equation \eqref{HJ} and
                \begin{equation*}
                    \varphi(T, x(\cdot))
                    \leq \sigma(x(\cdot)),
                    \quad \forall x(\cdot) \in C([- h, T], \mathbb{R}^n).
                \end{equation*}

            \smallskip

            \noindent (iii)
                We call $\varphi$ a {\it minimax solution} (resp., a {\it viscosity solution}) of the Cauchy problem \eqref{HJ}, \eqref{boundary_condition} if $\varphi$ is an upper as well as a lower minimax (resp., an upper as well as a lower viscosity) solution of this problem.
        \end{definition}

        In addition to Assumption \ref{assumption_H1_H2}, we make
        \begin{assumption} \label{assumption_H3_sigma}
            The Hamiltonian $H$ and the boundary functional $\sigma$ from \eqref{HJ} and \eqref{boundary_condition} satisfy the following conditions:

            \smallskip

            \noindent (i)
                For any compact set $X \subset G$, there exists $\lambda_H > 0$ such that
                \begin{equation*}
                    |H(t, x(\cdot), s) - H(t, y(\cdot), s)|
                    \leq \lambda_H (1 + \|s\|) \|x(\cdot) - y(\cdot)\|_{[- h, t]}
                \end{equation*}
                for all $(t, x(\cdot))$, $(t, y(\cdot)) \in X$ and all $s \in \mathbb{R}^n$.

            \smallskip

            \noindent (ii)
                The functional $\sigma$ is continuous.
        \end{assumption}

        The next theorem is a comparison result for upper and lower viscosity solutions.
        \begin{theorem} \label{theorem_comparison}
            Suppose that Assumptions \ref{assumption_H1_H2} and \ref{assumption_H3_sigma} hold.
            Let $\varphi_+$ be an upper and let $\varphi_-$ be a lower viscosity solution of the Cauchy problem \eqref{HJ}, \eqref{boundary_condition}.
            Then, the inequality below is valid:
            \begin{equation*}
                \varphi_+ (t, x(\cdot))
                \geq \varphi_- (t, x(\cdot)),
                \quad \forall (t, x(\cdot)) \in G.
            \end{equation*}
        \end{theorem}

        Theorem \ref{theorem_comparison} follows directly from Theorem \ref{theorem_main} and the corresponding comparison result for upper and lower minimax solutions of the Cauchy problem \eqref{HJ}, \eqref{boundary_condition} obtained as a part of the proof of \cite[Theorem 1]{Gomoyunov_Lukoyanov_Plaksin_2021}.

        In particular, Theorem \ref{theorem_comparison} yields a uniqueness result for viscosity solutions.
        \begin{corollary} \label{corollary_uniqueness}
            Let Assumptions \ref{assumption_H1_H2} and \ref{assumption_H3_sigma} hold.
            Then, there exists at most one viscosity solution of the Cauchy problem \eqref{HJ}, \eqref{boundary_condition}.
        \end{corollary}

        \begin{remark}
            The comparison theorems from \cite{Zhou_2015,Plaksin_2020_JOTA,Zhou_2020_2,Cosso_Gozzi_Rosestolato_Russo_2021,Plaksin_2021_SIAM,Zhou_2021} require that the correspondingly defined upper and lower viscosity solutions satisfy certain additional conditions like growth estimates, uniform continuity, or boundedness.
            In contrast to these results, Theorem \ref{theorem_comparison} assumes only that $\varphi_+ \in \Phi_+$ and $\varphi_- \in \Phi_-$, i.e., these functionals $\varphi_+$ and $\varphi_-$ are only semi-continuous (see Section \ref{subsection_minimax}).
            As a consequence, we conclude that Corollary \ref{corollary_uniqueness} ensures the uniqueness of a viscosity solution in a quite wide set of functionals $\Phi_+ \cap \Phi_-$.
            In addition, we note that Assumptions \ref{assumption_H1_H2} and \ref{assumption_H3_sigma} are weaker than those considered in the cited papers.
        \end{remark}

        Concerning an existence result for a viscosity solution of the Cauchy problem \eqref{HJ}, \eqref{boundary_condition}, the reader is referred to Section \ref{subsection_optimal_control_problem} below.

    \subsection{Some other criteria for minimax and viscosity solutions}

        In the theory of HJ equations with partial derivatives, various criteria for minimax and viscosity solutions are known (see, e.g., \cite[Sections 4.4 and 6.3]{Subbotin_1995} and \cite{Crandall_Lions_1983,Crandall_Evans_Lions_1984}).
        Theorem \ref{theorem_main} provides an analog of one of these criteria for the path-dependent HJ equation \eqref{HJ}.
        In this section we show that, in addition, Theorem \ref{theorem_main} allows us to immediately obtain analogues of some other criteria.
        For brevity, we deal with upper solutions only, while the corresponding results for lower solutions can be derived similarly.
        We first consider a general case when $\varphi \in \Phi_+$ (as in Definitions \ref{definition_minimax} and \ref{definition_viscosity}), and then we focus on a particular case when $\varphi$ satisfies an additional local Lipschitz condition.

        \subsubsection{General case}

            Let us first recall an infinitesimal criterion for an upper minimax solution of the path-dependent HJ equation \eqref{HJ}, obtained in \cite[Theorem 8.1]{Lukoyanov_2003_1} (see also \cite[Theorem 1]{Lukoyanov_2001_PMM_Eng} and \cite[Proposition 4]{Gomoyunov_Lukoyanov_Plaksin_2021}).

            Let $\varphi \colon G \to \mathbb{R}$, let $(t, x(\cdot)) \in G_0$, and let $L \subset \mathbb{R}^n$ be a non-empty convex compact set.
            The {\it lower right derivative} of the functional $\varphi$ at the point $(t, x(\cdot))$ in the {\it multi-valued direction} $L$ is defined by
            \begin{equation} \label{derivatives_multi-valued}
                \rd^- \varphi(t, x(\cdot); L)
                \triangleq \lim_{\varepsilon \downarrow 0} \inf_{\omega(\cdot) \in \Omega(t, x(\cdot), [L]^\varepsilon)}
                \partial^- \varphi(t, x(\cdot); \omega(\cdot)),
            \end{equation}
            where
            \begin{equation} \label{Omega}
                \Omega(t, x(\cdot), [L]^\varepsilon)
                \triangleq \bigl\{ \omega(\cdot) \in \Lip(t, x(\cdot)) \bigm|
                \dot{\omega}(\tau) \in [L]^\varepsilon \text{ for a.e. } \tau \in [t, T] \bigr\},
            \end{equation}
            the symbol $[L]^\varepsilon$ stands for the closed $\varepsilon$-neighborhood of $L$ in $\mathbb{R}^n$, and $\partial^- \varphi(t, x(\cdot); \omega(\cdot))$ is the lower right derivative along the function $\omega(\cdot)$ (see \eqref{derivatives_along_extenstion}).

            \begin{proposition} \label{proposition_criteria_minimax}
                Suppose that Assumption \ref{assumption_H1_H2} holds.
                Then, for every lower semi-continuous functional $\varphi \colon G \to \mathbb{R}$, condition \eqref{upper_minimax} is equivalent to the following one:
                \begin{equation} \label{upper_minimax_derivatives_multi-valued}
                    \rd^- \varphi_s \bigl( t, x(\cdot); B_{c_H}(t, x(\cdot)) \bigr) + H(t, x(\cdot), s)
                    \leq 0,
                    \quad \forall (t, x(\cdot)) \in G_0, \quad \forall s \in \mathbb{R}^n.
                \end{equation}
            \end{proposition}

            In condition \eqref{upper_minimax_derivatives_multi-valued}, the functional $\varphi_s$ is defined by $\varphi$ and $s$ according to \eqref{varphi_s}, and the set $B_{c_H}(t, x(\cdot))$ is introduced in \eqref{B}.

            Now, let us provide a similar criterion for an upper viscosity solution of \eqref{HJ}.
            \begin{proposition} \label{proposition_criteria_viscosity}
                For every functional $\varphi \colon G \to \mathbb{R}$, condition \eqref{upper_viscosity} is equivalent to the following one:
                \begin{equation} \label{upper_viscosity_derivatives_along_extension}
                    \inf_{z(\cdot) \in \Lip(t, x(\cdot))} \partial^- \varphi_s (t, x(\cdot); z(\cdot)) + H(t, x(\cdot), s)
                    \leq 0,
                    \quad \forall (t, x(\cdot)) \in G_0, \quad \forall s \in \mathbb{R}^n.
                \end{equation}
            \end{proposition}

            The proof is given in \ref{appendix_proofs_1}.

            Note that \eqref{upper_minimax_derivatives_multi-valued} implies \eqref{upper_viscosity_derivatives_along_extension}, while the reverse implication is not trivial.
            However, Theorem \ref{theorem_main} allows us to conclude that, under Assumption \ref{assumption_H1_H2}, these conditions are actually equivalent for every functional $\varphi \in \Phi_+$.

            In addition, it seems interesting to obtain also a criterion for condition \eqref{upper_viscosity} in terms of $ci$-smooth test functionals.
            \begin{proposition} \label{proposition_test_functionals}
                For every functional $\varphi \colon G \to \mathbb{R}$, condition \eqref{upper_viscosity} is equivalent to the following one:
                for every point $(t, x(\cdot)) \in G_0$ and every $ci$-smooth functional $\psi \colon G \to \mathbb{R}$, if, for any function $z(\cdot) \in \Lip(t, x(\cdot))$, there exists $\delta_z \in (0, T - t]$ such that
                \begin{equation} \label{upper_viscosity_test_functions_condition}
                    \varphi(t, x(\cdot)) - \psi(t, x(\cdot))
                    \leq \varphi(\tau, z_\tau(\cdot)) - \psi(\tau, z_\tau(\cdot)),
                    \quad \forall \tau \in [t, t + \delta_z],
                \end{equation}
                then the inequality below is fulfilled:
                \begin{equation} \label{upper_viscosity_test_functions_statement}
                    \partial_t \psi(t, x(\cdot)) + H \bigl( t, x(\cdot), \nabla \psi(t, x(\cdot)) \bigr)
                    \leq 0.
                \end{equation}
            \end{proposition}

            The proof is given in \ref{appendix_proofs_1}.

            Note that the condition formulated in Proposition \ref{proposition_test_functionals} requires the point $(t, x(\cdot))$ to be a local minimum of the difference between the functional $\varphi$ and the test functional $\psi$ only along every function $z(\cdot) \in \Lip(t, x(\cdot))$, which correlates with the definition of $ci$-differentiability (see \eqref{ci_derivatives}).

            Summarizing, we arrive at
            \begin{corollary} \label{corollary_criteria}
                Under Assumption \ref{assumption_H1_H2}, conditions \eqref{upper_minimax}, \eqref{upper_viscosity}, \eqref{upper_minimax_derivatives_multi-valued}, \eqref{upper_viscosity_derivatives_along_extension}, and the condition from Proposition \ref{proposition_test_functionals} are equivalent for every functional $\varphi \in \Phi_+$.
            \end{corollary}

            \begin{remark}
                Based on the idea of \cite{Soner_1988}, another approach to the development of the theory of viscosity solutions of the Cauchy problem \eqref{HJ}, \eqref{boundary_condition} was proposed in \cite{Lukoyanov_2007_IMM_Eng}.
                In that paper, the definition of a viscosity solution is given in terms of $ci$-smooth test functionals and involves a certain expanding sequence of compact subsets of $G$ which union is dense in $G$.
                Under mild assumptions, it is proved that a viscosity solution is a minimax solution (defined with the help of conditions \eqref{upper_minimax} and \eqref{lower_minimax}) and that a viscosity solution is unique.
                Relying on these facts and a uniqueness result for minimax solutions, it is obtained that the minimax and viscosity solutions coincide.
                However, apart from this connection via minimax solutions and uniqueness results, it is not clear how the definition of a viscosity solution from \cite{Lukoyanov_2007_IMM_Eng} and that introduced in the present paper relate to each other.
            \end{remark}

        \subsubsection{Case of Lipschitz continuous functionals}
        \label{subsubsection_Lipschitz}

            In this section, we concretize infinitesimal conditions \eqref{upper_viscosity} and \eqref{upper_viscosity_derivatives_along_extension} in a particular case when $\varphi$ satisfies an additional local Lipschitz condition in the second variable with respect to the auxiliary metric $\rho_1$ from \eqref{rho_1}.

            Let us denote by $\Phi_{\Lip}$ the set of all functionals $\varphi \colon G \to \mathbb{R}$ that satisfy the following {\it local Lipschitz condition}:
            for any $\alpha > 0$, there exists $\lambda_\varphi > 0$ such that
            \begin{equation} \label{Phi_Lip}
                |\varphi(t, x(\cdot)) - \varphi(t, y(\cdot))|
                \leq \lambda_\varphi \rho_1 \bigl( (t, x(\cdot)), (t, y(\cdot)) \bigr),
                \quad (t, x(\cdot)), (t, y(\cdot)) \in G(\alpha).
            \end{equation}

            Let $\varphi \colon G \to \mathbb{R}$, $(t, x(\cdot)) \in G_0$, and $l \in \mathbb{R}^n$.
            The {\it lower right derivative} of the functional $\varphi$ at the point $(t, x(\cdot))$ in the (single-valued) {\it direction} $l$ is defined by
            \begin{equation*}
                \partial^-_\ast \varphi(t, x(\cdot); l)
                \triangleq \partial^- \varphi(t, x(\cdot); z^{[l]}(\cdot)),
            \end{equation*}
            where the function $z^{[l]}(\cdot) \in \Lip(t, x(\cdot))$ is given by
            \begin{equation} \label{z^l}
                z^{[l]}(\tau)
                \triangleq x(t) + (\tau - t) l,
                \quad \forall \tau \in [t, T],
            \end{equation}
            and $\partial^- \varphi(t, x(\cdot); z^{[l]}(\cdot))$ is the lower right derivative along the function $z^{[l]}(\cdot)$ (see \eqref{derivatives_along_extenstion}).

            Let us first recall that, according to \cite[Section 5.4]{Lukoyanov_2006_IMM_Eng} (see also \cite[Proposition 4]{Lukoyanov_Plaksin_2019_MIAN_Eng}), for every functional $\varphi \in \Phi_{\Lip}$, every point $(t, x(\cdot)) \in G_0$, and every non-empty convex compact set $L \subset \mathbb{R}^n$, the equality below holds:
            \begin{equation} \label{derivatives_multivalued_via_single_valued}
                \rd^- \varphi(t, x(\cdot); L)
                = \inf_{l \in L} \partial^-_\ast \varphi(t, x(\cdot); l).
            \end{equation}
            As a consequence, for every $\varphi \in \Phi_{\Lip}$, condition \eqref{upper_minimax_derivatives_multi-valued} is equivalent to the condition
            \begin{equation} \label{condition_Lip_old}
                \begin{array}{c}
                    \displaystyle
                    \inf_{l \in B_{c_H}(t, x(\cdot))} \bigl( \partial^-_\ast \varphi (t, x(\cdot); l) - \langle s, l \rangle \bigr)
                    + H(t, x(\cdot), s)
                    \leq 0, \\[0.5em]
                    \forall (t, x(\cdot)) \in G_0, \quad \forall s \in \mathbb{R}^n.
                \end{array}
            \end{equation}

            Further, let us note the following property.
            \begin{proposition} \label{proposition_finite_derivatives}
                Let $\varphi \in \Phi_{\Lip}$ and $(t, x(\cdot)) \in G_0$.
                Then, for every $z(\cdot) \in \Lip(t, x(\cdot))$, there exists $l \in \mathbb{R}^n$ such that
                \begin{equation} \label{proposition_finite_derivatives_0}
                    \partial^-_\ast \varphi (t, x(\cdot); l)
                    \leq \partial^- \varphi (t, x(\cdot); z(\cdot)).
                \end{equation}
            \end{proposition}

            The proof is given in \ref{appendix_proofs_2}.

            Applying Proposition \ref{proposition_finite_derivatives}, we derive that, for every $\varphi \in \Phi_{\Lip}$ and every $(t, x(\cdot)) \in G_0$, the equalities below take place:
            \begin{equation*}
                \inf_{z(\cdot) \in \Lip(t, x(\cdot))} \partial^- \varphi (t, x(\cdot); z(\cdot))
                = \inf_{l \in \mathbb{R}^n} \partial^-_\ast \varphi (t, x(\cdot); l)
            \end{equation*}
            and
            \begin{equation} \label{subdifferential_finite}
                D^- \varphi (t, x(\cdot))
                = \bigl\{ (p_0, p) \in \mathbb{R} \times \mathbb{R}^n \bigm|
                p_0 + \langle p, l \rangle
                \leq \partial^-_\ast \varphi (t, x(\cdot); l),
                \, \forall l \in \mathbb{R}^n \bigr\}.
            \end{equation}
            Thus, by Corollary \ref{corollary_criteria}, we have
            \begin{theorem} \label{theorem_finite}
                Let Assumption \ref{assumption_H1_H2} hold and let $\varphi \in \Phi_+ \cap \Phi_{\Lip}$.
                Then, condition \eqref{upper_minimax} is equivalent to each of the following two conditions:

                \smallskip

                \noindent {\rm (i)}
                    The inequality below is valid:
                    \begin{equation} \label{condition_Lip_new}
                        \inf_{l \in \mathbb{R}^n} \bigl( \partial^-_\ast \varphi (t, x(\cdot); l) - \langle s, l \rangle \bigr)
                        + H(t, x(\cdot), s)
                        \leq 0,
                        \quad \forall (t, x(\cdot)) \in G_0, \quad \forall s \in \mathbb{R}^n.
                    \end{equation}

                \smallskip

                \noindent {\rm (ii)}
                    Condition \eqref{upper_viscosity} holds with the $ci$-subdifferential $D^- \varphi(t, x(\cdot))$ calculated by \eqref{subdifferential_finite}.
            \end{theorem}

            In other words, if we want to verify infinitesimal conditions \eqref{upper_viscosity} and \eqref{upper_viscosity_derivatives_along_extension} for a specific functional $\varphi \in \Phi_{\Lip}$, we do not need to consider all functions $z(\cdot)$ from $\Lip(t, x(\cdot))$, and we can restrict ourselves only to the (finite-dimensional) set of the functions $z^{[l]}(\cdot)$ for all $l \in \mathbb{R}^n$ (see \eqref{z^l}).

            Note that condition \eqref{condition_Lip_new} clearly follows from condition \eqref{condition_Lip_old}, but, in order to obtain the reverse implication, we use Corollary \ref{corollary_criteria} of Theorem \ref{theorem_main}.

            \begin{remark}
                In \cite{Plaksin_2020_JOTA}, to get a result similar to Theorem \ref{theorem_finite}, it was suggested to consider the HJ equation \eqref{HJ}, as well as all conditions used in the definitions of minimax and viscosity solutions, in the wider space of all piecewise continuous functions.
                The proof technique used in the present paper does not require such an extension of the space $G$.
            \end{remark}

            An existence result for a viscosity solution of the Cauchy problem \eqref{HJ}, \eqref{boundary_condition} that belongs to $\Phi_{\Lip}$ is given in the next section.

    \subsection{Application to optimal control problems}
    \label{subsection_optimal_control_problem}

        Path-dependent HJ equations of form \eqref{HJ} arise from optimal control problems and differential games for time-delay systems (see, e.g., \cite{Kim_1999,Lukoyanov_2000_JAMM,Lukoyanov_2003_2,Lukoyanov_2010_IMM_Eng_1,Kaise_2015,Bayraktar_Keller_2018}).
        In this regard, Theorem \ref{theorem_main} can be applied to characterize the value functionals in such problems under quite general assumptions.
        The goal of this section is to illustrate this fact and, in particular, to show that the requirements $\varphi \in \Phi_-$ and $\varphi \in \Phi_+$ in Definitions \ref{definition_minimax} and \ref{definition_viscosity} are not too restrictive.

        Let us consider the {\it optimal control problem} described by an {\it initial data} $(t, x(\cdot)) \in G$, the {\it dynamic equation}
        \begin{equation} \label{system}
            \dot{z}(\tau)
            = f(\tau, z_\tau(\cdot), u(\tau))
            \text{ for a.e. } \tau \in [t, T]
        \end{equation}
        under the {\it initial condition}
        \begin{equation} \label{initial_condition}
            z_t(\cdot)
            = x(\cdot)
        \end{equation}
        and the {\it cost functional}
        \begin{equation} \label{cost_functional}
            J
            \triangleq \sigma(z(\cdot)) + \int_{t}^{T} \chi(\tau, z_\tau(\cdot), u(\tau)) \, \rd \tau
        \end{equation}
        to be minimized.
        Here, $\tau$ is time, $z(\tau) \in \mathbb{R}^n$ is the current state, $u(\tau) \in U$ is the current control, $U \subset \mathbb{R}^m$ is a compact set, $m \in \mathbb{N}$, and $z_\tau(\cdot)$ is the history of the system motion $z(\cdot)$ realized up to $\tau$ (see \eqref{x_t}).

        \begin{assumption} \label{assumption_optimal_control}
             The mappings $G \times U \ni ((t, x(\cdot)), u) \mapsto f(t, x(\cdot), u) \in \mathbb{R}^n$ from \eqref{system} and $\sigma: C([- h, T], \mathbb{R}^n) \to \mathbb{R}$ and $G \times U \ni ((t, x(\cdot)), u) \mapsto \chi(t, x(\cdot), u) \in \mathbb{R}$ from \eqref{cost_functional} satisfy the following conditions:

            \smallskip

            \noindent (i)
                The mappings $f$ and $\chi$ are continuous.

            \smallskip

            \noindent {\rm (ii)}
                There exists $c_{f, \chi} > 0$ such that, for all $(t, x(\cdot)) \in G$ and all $u \in U$,
                \begin{equation*}
                    \|f(t, x(\cdot), u)\|
                    + |\chi(t, x(\cdot), u)|
                    \leq c_{f, \chi} (1 + \|x(\cdot)\|_{[- h, t]}).
                \end{equation*}

            \smallskip

            \noindent {\rm (iii)}
                For any $\alpha > 0$, there exists $\lambda_{f, \chi} > 0$ such that
                \begin{multline*}
                    \|f(t, x(\cdot), u) - f(t, y(\cdot), u)\| + |\chi(t, x(\cdot), u) - \chi(t, y(\cdot), u)| \\
                    \leq \lambda_{f, \chi} \biggl( \|x(t) - y(t)\| + \|x(t - h) - y(t - h)\| + \int_{- h}^{t} \|x(\xi) - y(\xi)\| \, \rd \xi \biggr)
                \end{multline*}
                for all $(t, x(\cdot))$, $(t, y(\cdot)) \in G(\alpha)$ (see \eqref{G_alpha}) and all $u \in U$.

            \smallskip

            \noindent {\rm (iv)}
                For every $\alpha > 0$ and every $\varepsilon > 0$, there exists $\delta_\sigma > 0$ such that
                \begin{equation*}
                    |\sigma(x(\cdot)) - \sigma(y(\cdot))|
                    \leq \varepsilon
                \end{equation*}
                for all $(T, x(\cdot))$, $(T, y(\cdot)) \in G(\alpha)$ with $\rho_1 ((T, x(\cdot)), (T, y(\cdot))) \leq \delta_\sigma$ (see \eqref{rho_1}).
        \end{assumption}

        The set of {\it admissible controls} $\mathcal{U}(t)$ consists of all measurable functions $u \colon [t, T] \to U$.
        In view of conditions (i)--(iii) of Assumption \ref{assumption_optimal_control}, for every $u(\cdot) \in \mathcal{U}(t)$, there exists a unique {\it system motion} $z(\cdot) \triangleq z(\cdot; t, x(\cdot), u(\cdot))$, which is defined as a function from $\Lip(t, x(\cdot))$ that together with $u(\cdot)$ satisfies the dynamic equation \eqref{system} for a.e. $\tau \in [t, T]$.
        Let $J(t, x(\cdot), u(\cdot))$ be the corresponding value of the cost functional \eqref{cost_functional}.

        Then, the {\it value functional} $\varphi_0 \colon G \to \mathbb{R}$ is given by
        \begin{equation} \label{value}
            \varphi_0 (t, x(\cdot))
            \triangleq \inf_{u(\cdot) \in \mathcal{U}(t)} J(t, x(\cdot), u(\cdot)),
            \quad \forall (t, x(\cdot)) \in G.
        \end{equation}

        According to, e.g., \cite{Lukoyanov_2010_IMM_Eng_1}, the Bellman equation associated to the optimal control problem \eqref{system}--\eqref{cost_functional} is the path-dependent HJ equation \eqref{HJ} with the Hamiltonian
        \begin{equation} \label{H}
            H(t, x(\cdot),s)
            \triangleq \min_{u \in U} \bigl( \langle s, f(t, x(\cdot), u) \rangle + \chi(t, x(\cdot), u) \bigr),
            \quad \forall (t, x(\cdot)) \in G, \quad \forall s \in \mathbb{R}^n.
        \end{equation}
        In particular, it follows from \cite[Theorem 2]{Lukoyanov_2010_IMM_Eng_1} that, under Assumption \ref{assumption_optimal_control}, the value functional $\varphi_0$ is continuous and satisfies conditions \eqref{upper_minimax} and \eqref{lower_minimax} with $H$ from \eqref{H}.
        By construction, $\varphi_0$ meets the boundary condition \eqref{boundary_condition}.
        Moreover, we have
        \begin{proposition} \label{proposition_value_Phi}
            Under Assumption \ref{assumption_optimal_control}, the inclusion $\varphi_0 \in \Phi_+ \cap \Phi_-$ holds.
        \end{proposition}

        The proof is given in \ref{appendix_proofs_3}.

        Thus, we conclude that $\varphi_0$ is a minimax solution of the associated Cauchy problem \eqref{HJ}, \eqref{boundary_condition}.
        Therefore, and since Assumption \ref{assumption_optimal_control} implies that the Hamiltonian $H$ from \eqref{H} and the boundary functional $\sigma$ satisfy Assumptions \ref{assumption_H1_H2} and \ref{assumption_H3_sigma}, we obtain that $\varphi_0$ is a viscosity solution of this problem by Theorem \ref{theorem_main}.
        In addition, such a viscosity solution is unique due to Corollary \ref{corollary_uniqueness}.
        As a result, we get
        \begin{theorem} \label{theorem_value_is_viscosity_solution}
            Under Assumption \ref{assumption_optimal_control}, a functional $\varphi \colon G \to \mathbb{R}$ is the value functional of the optimal control problem \eqref{system}--\eqref{cost_functional} if and only if $\varphi$ is a viscosity solution of the Cauchy problem for the associated Bellman equation \eqref{HJ}, where $H$ is given by \eqref{H}, and the boundary condition \eqref{boundary_condition}.
        \end{theorem}

        Finally, let us consider the following condition on the functional $\sigma$, which is stronger than condition (iv) of Assumption \ref{assumption_optimal_control}.
        \begin{assumption} \label{assumption_sigma_Lipschitz_continuous}
            For every $\alpha > 0$, there exists $\lambda_\sigma > 0$ such that
            \begin{equation*}
                |\sigma(x(\cdot)) - \sigma(y(\cdot))|
                \leq \lambda_\sigma \rho_1 \bigl( (T, x(\cdot)),(T, y(\cdot)) \bigr),
                \quad \forall (T, x(\cdot)), (T, y(\cdot)) \in G(\alpha).
            \end{equation*}
        \end{assumption}

        In this case, we have
        \begin{proposition} \label{proposition_existence_Lip}
            Let conditions {\rm(i)}--{\rm(iii)} of Assumption \ref{assumption_optimal_control} and Assumption \ref{assumption_sigma_Lipschitz_continuous} hold.
            Then, the inclusion $\varphi_0 \in \Phi_{\Lip}$ is fulfilled.
        \end{proposition}

        The proof is given in \ref{appendix_proofs_3}.

        From Theorem \ref{theorem_value_is_viscosity_solution} and Proposition \ref{proposition_existence_Lip}, we derive
        \begin{corollary} \label{corollary_value_is_viscosity_solution_from_Lip}
            Let conditions {\rm(i)}--{\rm(iii)} of Assumption \ref{assumption_optimal_control} and Assumption \ref{assumption_sigma_Lipschitz_continuous} hold.
            Then, a functional $\varphi \colon G \to \mathbb{R}$ is the value functional of the optimal control problem \eqref{system}--\eqref{cost_functional} if and only if $\varphi \in \Phi_{\Lip}$ and $\varphi$ is a viscosity solution of the Cauchy problem for the associated Bellman equation \eqref{HJ}, where $H$ is given by \eqref{H}, and the boundary condition \eqref{boundary_condition}.
        \end{corollary}

        In particular, Theorem \ref{theorem_value_is_viscosity_solution} gives us a sufficient condition for existence of a viscosity solution of the Cauchy problem \eqref{HJ}, \eqref{boundary_condition}, while Corollary \ref{corollary_value_is_viscosity_solution_from_Lip} provides a sufficient condition for existence of a viscosity solution $\varphi \colon G \to \mathbb{R}$ of this problem satisfying the additional requirement $\varphi \in \Phi_{\Lip}$.

\section{Proof of Theorem \ref{theorem_main}}
\label{section_proof}

    This section is devoted to the proof of Theorem \ref{theorem_main}.
    Note that we prove only the first part of the theorem concerning upper solutions.
    The scheme of the proof is the following.
    In Section \ref{subsection_variational_principle}, we establish a variational principle, which is based on Borwein--Preiss variational principle from \cite[Theorem 1]{Li_Shi_2000}.
    In Section \ref{subsection_Clarke_Ledyaev}, we apply the variational principle to prove a certain property of the $ci$-subdifferential (see \eqref{subdifferential_via_directional}), which can be seen as an analog of Property of Subdifferential \cite[Theorem 1.1]{Subbotin_1993_MSb} (see also \cite[Theorem 3.1]{Clarke_Ledyaev_1994} and \cite[Section A6]{Subbotin_1995}).
    In Section \ref{subsection_proof}, we provide three lemmas that imply the first part of Theorem \ref{theorem_main}.
    The second part, which concerns lower solutions, can be proved similarly, while the third part is a direct consequence of the two previous ones.

    \subsection{Variational principle}
    \label{subsection_variational_principle}

        For every $\alpha > 0$, let us denote
        \begin{equation} \label{A_alpha}
            c_\alpha
            \triangleq T^2 + 8 \alpha^2 + 1
            > 0.
        \end{equation}

        The goal of this section is to prove
        \begin{lemma} \label{lemma_variational_principle_G}
            Let $\alpha > 0$, let $X \subset G$ be a non-empty closed set such that $X \subset G(\alpha) \cap G_0$, and let $\varphi \colon X \to \mathbb{R}$ be a functional that is lower semi-continuous and bounded from below.
            Then, for every $\varkappa \in (0, 1]$, there exist a functional $\psi \colon G \to \mathbb{R}$ and a point $(t^\ast, x^\ast(\cdot)) \in X$ such that the following statements hold:

            \smallskip

            \noindent {\rm(i)}
                The inequality below is valid:
                \begin{equation} \label{lemma_variational_principle_G_part_i}
                    |\psi(t, x(\cdot))|
                    \leq 2 c_\alpha \varkappa,
                    \quad \forall (t, x(\cdot)) \in X.
                \end{equation}

            \smallskip

            \noindent {\rm(ii)}
                The functional $\psi$ is $ci$-differentiable at the point $(t^\ast, x^\ast(\cdot))$, and the corresponding $ci$-derivatives satisfy the estimates
                \begin{equation} \label{lemma_variational_principle_G_part_ii}
                    |\partial_t \psi(t^\ast, x^\ast(\cdot))|
                    \leq 4 T \varkappa,
                    \quad \| \nabla \psi(t^\ast, x^\ast(\cdot)) \|
                    \leq 8 \alpha \varkappa.
                \end{equation}

            \smallskip

            \noindent {\rm(iii)}
                It holds that
                \begin{equation} \label{lemma_variational_principle_G_part_iii}
                    \varphi(t^\ast, x^\ast(\cdot)) + \psi(t^\ast, x^\ast(\cdot))
                    = \min_{(t, x(\cdot)) \in X} \bigl( \varphi(t, x(\cdot)) + \psi(t, x(\cdot)) \bigr).
                \end{equation}
        \end{lemma}

        We prove this lemma on the basis of the variational principle from \cite[Theorem 1]{Li_Shi_2000}.
        In order to apply this principle, we first need to construct a guage-type functional.

        \subsubsection{Guage-type functional}

        Following \cite{Zhou_2020_1}, let us consider the functional $V \colon G \to \mathbb{R}$ given by
        \begin{equation} \label{V}
            V(t, x(\cdot))
            \triangleq
            \begin{cases}
                \displaystyle
                \frac{\bigl( \|x(\cdot)\|_{[- h, t]}^2 - \|x(t)\|^2 \bigr)^2}{\|x(\cdot)\|_{[- h, t]}^2} + \|x(t)\|^2,
                & \mbox{if } \|x(\cdot)\|_{[- h, t]} > 0, \\
                0,
                & \mbox{if } \|x(\cdot)\|_{[- h, t]} = 0,
            \end{cases}
        \end{equation}
        for all $(t, x(\cdot)) \in G$.
        Recall that the notation $\|x(\cdot)\|_{[- h, t]}$ is introduced in \eqref{norm}.
        By \cite[Lemma 2.3]{Zhou_2020_1} (see also \cite[Section 4 and Appendix B]{Gomoyunov_Lukoyanov_Plaksin_2021}), the functional $V$ is $ci$-smooth and its $ci$-derivatives are as follows:
        \begin{equation} \label{V_ci_derivative_t}
            \partial_t V(t, x(\cdot))
            = 0
        \end{equation}
        and
        \begin{equation*}
            \nabla V (t, x(\cdot))
            = \begin{cases}
                \displaystyle
                \biggl( 2 - \frac{4 \bigl( \|x(\cdot)\|_{[- h, t]}^2 - \|x(t)\|^2 \bigr)}{\|x(\cdot)\|_{[- h, t]}^2} \biggr) x(t),
                & \mbox{if } \|x(\cdot)\|_{[- h, t]} > 0, \\
                0,
                & \mbox{if } \|x(\cdot)\|_{[- h, t]} = 0,
            \end{cases}
        \end{equation*}
        for all $(t, x(\cdot)) \in G_0$.
        Moreover, the estimates below hold:
        \begin{equation} \label{V_estimates}
            \frac{3 - \sqrt{5}}{2} \|x(\cdot)\|_{[- h, t]}^2
            \leq V(t, x(\cdot))
            \leq 2 \|x(\cdot)\|_{[- h, t]}^2,
            \quad \forall (t, x(\cdot)) \in G,
        \end{equation}
        and
        \begin{equation} \label{nabla_V_estimates}
            \|\nabla V (t, x(\cdot))\|
            \leq 2 \|x(t)\|,
            \quad \forall (t, x(\cdot)) \in G_0.
        \end{equation}

        For every $\alpha > 0$, let us define the functional $\mu_\alpha \colon G \times G \to \mathbb{R}$ by
        \begin{equation} \label{varrho_alpha}
            \mu_\alpha \bigl( (t, x(\cdot)), (\tau, y(\cdot)) \bigr)
            \triangleq
            \begin{cases}
                (t - \tau)^2 + V(t, x(\cdot) - y_t(\cdot \wedge \tau)), & \mbox{if } t \geq \tau, \\
                c_\alpha, & \mbox{if } t < \tau,
            \end{cases}
        \end{equation}
        for all $(t, x(\cdot))$, $(\tau, y(\cdot)) \in G$.
        Here, $c_\alpha$ is taken from \eqref{A_alpha}, and, in accordance with \eqref{x_t}, the function $y_t(\cdot \wedge \tau)$ is the restriction of the function $y(\cdot \wedge \tau)$ to the interval $[- h, t]$.

        The next lemma shows, in particular, that the functional $\mu_\alpha$ meets the requirements for a guage-type functional from \cite[Theorem 1]{Li_Shi_2000}.
        \begin{lemma} \label{lemma_guage-type_function}
            For every $\alpha > 0$, the functional $\mu_\alpha$ has the following properties:

            \smallskip

            \noindent {\rm(i)}
                The functional $\mu_\alpha$ is non-negative, and the estimate
                \begin{equation} \label{varrho.1}
                    \mu_\alpha \bigl( (t, x(\cdot)), (\tau, y(\cdot)) \bigr)
                    \leq T^2 + 2 (\|x(\cdot)\|_{[- h, t]} + \|y(\cdot)\|_{[- h, \tau]})^2
                \end{equation}
                is valid for all $(t, x(\cdot))$, $(\tau, y(\cdot)) \in G$ with $t \geq \tau$.

            \smallskip

            \noindent {\rm(ii)}
                For every $(t, x(\cdot)) \in G$, the equality $\mu_\alpha ((t, x(\cdot)), (t, x(\cdot))) = 0$ takes place.

            \smallskip

            \noindent {\rm(iii)}
                Let sequences $\{(t^{(k)}, x^{(k)}(\cdot))\}_{k = 1}^\infty$, $\{(\tau^{(k)}, y^{(k)}(\cdot))\}_{k = 1}^\infty \subset G$ be such that
                \begin{equation} \label{varrho.3_part_1}
                    \mu_\alpha \bigl( (t^{(k)}, x^{(k)}(\cdot)), (\tau^{(k)}, y^{(k)}(\cdot)) \bigr)
                    \to 0
                    \text{ as } k \to \infty.
                \end{equation}
                Then, it holds that
                \begin{equation} \label{varrho.3_part_2}
                    \rho_\infty \bigl( (t^{(k)}, x^{(k)}(\cdot)), (\tau^{(k)}, y^{(k)}(\cdot)) \bigr)
                    \to 0
                    \text{ as } k \to \infty.
                \end{equation}

            \smallskip

            \noindent {\rm(iv)}
                For every $(\tau, y(\cdot)) \in G(\alpha)$, the functional below is lower semi-continuous:
                \begin{equation*}
                    G(\alpha) \ni (t, x(\cdot)) \mapsto \mu_\alpha \bigl( (t, x(\cdot)), (\tau, y(\cdot)) \bigr) \in \mathbb{R}.
                \end{equation*}
        \end{lemma}
        \begin{proof}
            Properties (i) and (ii) follow directly from definition \eqref{varrho_alpha} of the functional $\mu_\alpha$, definition \eqref{V} of the functional $V$, and the second inequality in \eqref{V_estimates}.

            Let us prove property (iii).
            In view of \eqref{varrho.3_part_1}, we can assume that
            \begin{equation*}
                \mu_\alpha \bigl( (t^{(k)}, x^{(k)}(\cdot)), (\tau^{(k)}, y^{(k)}(\cdot)) \bigr)
                < c_\alpha,
                \quad \forall k \in \mathbb{N}.
            \end{equation*}
            Then, due to \eqref{varrho_alpha}, we find that $t^{(k)} \geq \tau^{(k)}$ and
            \begin{equation*}
                \mu_\alpha \bigl( (t^{(k)}, x^{(k)}(\cdot)), (\tau^{(k)}, y^{(k)}(\cdot)) \bigr)
                = (t^{(k)} - \tau^{(k)})^2 + V(t^{(k)}, x^{(k)}(\cdot) - y^{(k)}_{t^{(k)}}(\cdot \wedge \tau^{(k)}))
            \end{equation*}
            for all $k \in \mathbb{N}$.
            Hence, by \eqref{varrho.3_part_1}, we conclude that, as $k \to \infty$,
            \begin{equation} \label{proof_varrho.3_1}
                |t^{(k)} - \tau^{(k)}|
                \to 0,
                \quad V(t^{(k)}, x^{(k)}(\cdot) - y^{(k)}_{t^{(k)}}(\cdot \wedge \tau^{(k)}))
                \to 0.
            \end{equation}
            In addition, using the first inequality in \eqref{V_estimates}, we derive
            \begin{align}
                V(t^{(k)}, x^{(k)}(\cdot) - y^{(k)}_{t^{(k)}}(\cdot \wedge \tau^{(k)}))
                & \geq \frac{3 - \sqrt{5}}{2} \|x^{(k)}(\cdot) - y^{(k)}_{t^{(k)}}(\cdot \wedge \tau^{(k)})\|_{[- h, t^{(k)}]}^2
                \nonumber \\
                & = \frac{3 - \sqrt{5}}{2} \|x^{(k)}(\cdot \wedge t^{(k)}) - y^{(k)}(\cdot \wedge \tau^{(k)})\|_{[- h, T]}^2
                \label{proof_varrho.3_2}
            \end{align}
            for all $k \in \mathbb{N}$.
            From \eqref{proof_varrho.3_1} and \eqref{proof_varrho.3_2}, taking definition \eqref{rho_infty} of the metric $\rho_\infty$ into account, we derive \eqref{varrho.3_part_2}.

            Property (iv) can be verified directly.
            To this end, it suffices to fix $(\tau, y(\cdot)) \in G(\alpha)$ and observe that the functional
            \begin{equation*}
                G \ni (t, x(\cdot)) \mapsto (t - \tau)^2 + V(t, x(\cdot) - y_t(\cdot \wedge \tau)) \in \mathbb{R}
            \end{equation*}
            is continuous due to continuity of $V$ and that the inequality
            \begin{equation} \label{varrho_bound}
                \mu_\alpha \bigl( (t, x(\cdot)), (\tau, y(\cdot)) \bigr)
                \leq c_\alpha,
                \quad \forall (t, x(\cdot)) \in G(\alpha),
            \end{equation}
            holds owing to estimate \eqref{varrho.1}, definition \eqref{G_alpha} of $G(\alpha)$, and choice \eqref{A_alpha} of $c_\alpha$.
        \end{proof}

        \subsubsection{Proof of Lemma \ref{lemma_variational_principle_G}}

            Note that the metric space $(X, \rho_\infty)$ is complete since the metric space $(G, \rho_\infty)$ is complete and the set $X$ is closed.
            Then, due to the assumptions on the functional $\varphi$, we can apply \cite[Theorem 1]{Li_Shi_2000}, where we take the functional $\mu_\alpha$ from \eqref{varrho_alpha} (more precisely, its restriction to $X \times X$) as a guage-type functional in view of Lemma \ref{lemma_guage-type_function} and the inclusion $X \subset G(\alpha)$.
            As a result, we obtain that there exist a sequence $\{(\tau^{(k)}, y^{(k)}(\cdot))\}_{k = 0}^\infty \subset X$ and a point $(t^\ast, x^\ast(\cdot)) \in X$ such that
            \begin{equation} \label{varrho_alpha_bound_i}
                \mu_\alpha \bigl( (t^\ast, x^\ast(\cdot)), (\tau^{(k)}, y^{(k)}(\cdot)) \bigr)
                \leq \varkappa / 2^k,
                \quad \forall k \in \mathbb{N} \cup \{0\},
            \end{equation}
            and, for the functional $\psi \colon G \to \mathbb{R}$ given by
            \begin{equation} \label{psi_definition}
                \psi(t, x(\cdot))
                \triangleq \varkappa \sum_{k = 0}^{\infty} \frac{1}{2^k} \mu_\alpha \bigl( (t, x(\cdot)), (\tau^{(k)}, y^{(k)}(\cdot)) \bigr),
                \quad \forall (t, x(\cdot)) \in G,
            \end{equation}
            equality \eqref{lemma_variational_principle_G_part_iii} is valid.
            Observe that the functional $\psi$ is well-defined, i.e., the series in \eqref{psi_definition} converges for all $(t, x(\cdot)) \in G$.
            Indeed, for any $k \in \mathbb{N} \cup \{0\}$, by virtue of definition \eqref{varrho_alpha} of the functional $\mu_\alpha$ and inequality \eqref{varrho.1}, we have
            \begin{equation*}
                \mu_\alpha \bigl( (t, x(\cdot)), (\tau^{(k)}, y^{(k)}(\cdot)) \bigr)
                \leq \max \bigl\{ c_\alpha, T^2 + 2 (\|x(\cdot)\|_{[- h, t]} + \alpha)^2 \bigr\},
                \quad \forall (t, x(\cdot)) \in G.
            \end{equation*}
            In addition, for every $(t, x(\cdot)) \in X \subset G(\alpha)$, according to \eqref{varrho_bound}, we get
            \begin{equation*}
                0
                \leq \psi(t, x(\cdot))
                \leq c_\alpha \varkappa \sum_{k = 0}^{\infty} \frac{1}{2^k}
                = 2 c_\alpha \varkappa,
            \end{equation*}
            which proves \eqref{lemma_variational_principle_G_part_i}.

            It is important to note that, owing to estimate \eqref{varrho_alpha_bound_i}, the inequality $\varkappa \leq 1$, and definition \eqref{A_alpha} of $c_\alpha$, we have
            \begin{equation*}
                \mu_\alpha \bigl( (t^\ast, x^\ast(\cdot)), (\tau^{(k)}, y^{(k)}(\cdot)) \bigr)
                \leq 1
                < c_\alpha,
                \quad \forall k \in \mathbb{N} \cup \{0\},
            \end{equation*}
            and, therefore, in view of \eqref{varrho_alpha}, we obtain
            \begin{equation} \label{t_ast_geq_t_i}
                t^\ast
                \geq \tau^{(k)},
                \quad \forall k \in \mathbb{N} \cup \{0\}.
            \end{equation}

            Now, let us show that the functional $\psi$ is $ci$-differentiable at the point $(t^\ast, x^\ast(\cdot))$, belonging to $G_0$ since $X \subset G_0$, and that the corresponding $ci$-derivatives are given by
            \begin{equation} \label{psi_ci_derivatives_t}
                \partial_t \psi (t^\ast, x^\ast(\cdot))
                = \varkappa \sum_{k = 0}^{\infty} \frac{t^\ast - \tau^{(k)}}{2^{k - 1}}
            \end{equation}
            and
            \begin{equation} \label{psi_ci_derivatives_nabla}
                \nabla \psi (t^\ast, x^\ast(\cdot))
                = \varkappa \sum_{k = 0}^{\infty} \frac{1}{2^k} \nabla V(t^\ast, x^\ast(\cdot) - y^{(k)}_{t^\ast}(\cdot \wedge \tau^{(k)})).
            \end{equation}
            Let $z(\cdot) \in \Lip(t^\ast, x^\ast(\cdot))$.
            Let us choose $\lambda_z > 0$ such that $\|z(\tau) - z(\xi)\| \leq \lambda_z |\tau - \xi|$ for all $\tau$, $\xi \in [t^\ast, T]$.
            In particular, recalling that $(t^\ast, x^\ast(\cdot)) \in X \subset G(\alpha)$, we derive
            \begin{equation*}
                \|z(\tau)\|
                \leq \|x^\ast(t^\ast)\| + \lambda_z (T - t^\ast)
                \leq \alpha + \lambda_z (T - t^\ast),
                \quad \forall \tau \in [t^\ast, T],
            \end{equation*}
            and, hence, $\|z(\cdot)\|_{[- h, T]} \leq \alpha + \lambda_z (T - t^\ast)$.

            By \eqref{varrho_alpha} and \eqref{t_ast_geq_t_i}, we have
            \begin{equation} \label{psi_representation}
                \psi(\tau, z_\tau(\cdot))
                = \varkappa \sum_{k = 0}^{\infty} \frac{1}{2^k}
                \bigl( (\tau - \tau^{(k)})^2 + V(\tau, z_\tau(\cdot) - y^{(k)}_{\tau}(\cdot \wedge \tau^{(k)})) \bigr),
                \quad \forall \tau \in [t^\ast, T].
            \end{equation}

            Let us fix $k \in \mathbb{N} \cup \{0\}$ and consider the function
            \begin{equation*}
                z^{(k)}(\tau)
                \triangleq z(\tau) - y^{(k)}(\tau \wedge \tau^{(k)}),
                \quad \forall \tau \in [- h, T].
            \end{equation*}
            Since $z(\cdot) \in \Lip(t^\ast, x^\ast(\cdot))$ and $y^{(k)}(\tau \wedge \tau^{(k)}) = y^{(k)}(\tau^{(k)})$ for all $\tau \in [t^\ast, T]$  due to \eqref{t_ast_geq_t_i}, we obtain $z^{(k)}(\cdot) \in \Lip(t^\ast, x^\ast(\cdot) - y^{(k)}_{t^\ast} (\cdot \wedge \tau^{(k)}))$.
            Then, recalling that the functional $V$ is $ci$-\-smooth and equality \eqref{V_ci_derivative_t} holds, we get (see, e.g., \cite[Proposition 1]{Gomoyunov_Lukoyanov_Plaksin_2021})
            \begin{align}
                & V(\tau, z_\tau(\cdot) - y^{(k)}_\tau(\cdot \wedge \tau^{(k)})) - V(t^\ast, x^\ast(\cdot) - y^{(k)}_{t^\ast} (\cdot \wedge \tau^{(k)}))
                \nonumber \\
                & \quad = V(\tau, z^{(k)}_\tau(\cdot)) - V(t^\ast, x^\ast(\cdot) - y^{(k)}_{t^\ast} (\cdot \wedge \tau^{(k)}))
                \nonumber \\
                & \quad = \int_{t^\ast}^{\tau} \langle \nabla V(\xi, z^{(k)}_\xi(\cdot)), \dot{z}^{(k)}(\xi) \rangle \, \rd \xi
                \nonumber \\
                & \quad = \int_{t^\ast}^{\tau} \langle \nabla V(\xi, z_\xi(\cdot) - y^{(k)}_\xi (\cdot \wedge \tau^{(k)})), \dot{z}(\xi) \rangle \, \rd \xi,
                \quad \forall \tau \in [t^\ast, T).
                \label{V_integral}
            \end{align}

            From \eqref{psi_representation} and \eqref{V_integral}, it follows that
            \begin{multline*}
                \psi(\tau, z_\tau(\cdot)) - \psi(t^\ast, x^\ast(\cdot)) \\
                = \varkappa \sum_{k = 0}^{\infty} \frac{1}{2^k}
                \int_{t^\ast}^{\tau} \bigl( 2 (\xi - \tau^{(k)})
                + \langle \nabla V(\xi, z_\xi(\cdot) - y^{(k)}_\xi (\cdot \wedge \tau^{(k)})), \dot{z}(\xi) \rangle \bigr) \, \rd \xi
            \end{multline*}
            for all $\tau \in [t^\ast, T)$.
            Note that the functions
            \begin{equation} \label{auxiliary_functions}
                [t^\ast, T) \ni \xi
                \mapsto \nabla V(\xi, z_\xi(\cdot) - y^{(k)}_\xi (\cdot \wedge \tau^{(k)})) \in \mathbb{R}^n,
                \quad \forall k \in \mathbb{N} \cup \{0\},
            \end{equation}
            are continuous owing to $ci$-smoothness of $V$.
            In addition, for a.e. $\xi \in [t^\ast, T]$, using estimate \eqref{nabla_V_estimates}, the inclusions $(\tau^{(k)}, y^{(k)}(\cdot)) \in X \subset G(\alpha)$ for all $k \in \mathbb{N} \cup \{0\}$, and the choice of $\lambda_z$, we derive
            \begin{align}
                & |2 (\xi - \tau^{(k)})
                + \langle \nabla V(\xi, z_\xi(\cdot) - y^{(k)}_\xi (\cdot \wedge \tau^{(k)})), \dot{z}(\xi) \rangle|
                \nonumber \\
                & \quad \leq 2 T + 2 \lambda_z \|z(\xi) - y^{(k)}(\xi \wedge \tau^{(k)})\|
                \nonumber \\
                & \quad \leq 2 T + 2 \lambda_z (2 \alpha + \lambda_z (T - t^\ast)),
                \quad \forall k \in \mathbb{N} \cup \{0\}.
                \label{estimate_Lebesgue_basis}
            \end{align}
            Hence, applying the Lebesgue dominated convergence theorem, we get
            \begin{multline} \label{estimate_Lebesgue}
                \psi(\tau, z_\tau(\cdot)) - \psi(t^\ast, x^\ast(\cdot)) \\
                = \int_{t^\ast}^{\tau} \varkappa \sum_{k = 0}^{\infty} \frac{\xi - \tau^{(k)}}{2^{k - 1}} \, \rd \xi
                + \int_{t^\ast}^{\tau} \biggl\langle \varkappa \sum_{k = 0}^{\infty} \frac{1}{2^k} \nabla V(\xi, z_\xi(\cdot) - y^{(k)}_\xi (\cdot \wedge \tau^{(k)})),
                \dot{z}(\xi) \biggl\rangle \, \rd \xi
            \end{multline}
            for all $\tau \in [t^\ast, T)$.
            For the first integral in \eqref{estimate_Lebesgue}, in accordance with \eqref{psi_ci_derivatives_t}, we have
            \begin{equation} \label{estimate_Lebesgue_first}
                \frac{1}{\tau - t^\ast}
                \biggl| \int_{t^\ast}^{\tau} \varkappa \sum_{k = 0}^{\infty} \frac{\xi - \tau^{(k)}}{2^{k - 1}} \, \rd \xi
                - \partial_t \psi(t^\ast, x^\ast(\cdot)) (\tau - t^\ast) \biggr|
                \to 0
                \text{ as } \tau \downarrow t^\ast.
            \end{equation}
            Moreover, since functions \eqref{auxiliary_functions} are continuous and satisfy the estimates (see \eqref{estimate_Lebesgue_basis})
            \begin{equation*}
                \|\nabla V(\xi, z_\xi(\cdot) - y^{(k)}_\xi (\cdot \wedge \tau^{(k)}))\|
                \leq 2 (2 \alpha + \lambda_z (T - t^\ast)),
                \quad \forall \xi \in [t^\ast, T),
                \quad \forall k \in \mathbb{N} \cup \{0\},
            \end{equation*}
            we conclude that the function
            \begin{equation*}
                [t^\ast, T) \ni \xi
                \mapsto \varkappa \sum_{k = 0}^{\infty} \frac{1}{2^k} \nabla V(\xi, z_\xi(\cdot) - y^{(k)}_\xi (\cdot \wedge \tau^{(k)})) \in \mathbb{R}^n
            \end{equation*}
            is continuous.
            Consequently, for the second integral in \eqref{estimate_Lebesgue}, in view of \eqref{psi_ci_derivatives_nabla}, we obtain
            \begin{multline} \label{estimate_Lebesgue_second}
                \frac{1}{\tau - t^\ast} \biggl| \int_{t^\ast}^{\tau} \biggl\langle \varkappa
                \sum_{k = 0}^{\infty} \frac{1}{2^k} \nabla V(\xi, z_\xi(\cdot) - y^{(k)}_\xi (\cdot \wedge \tau^{(k)})),
                \dot{z}(\xi) \biggl\rangle \, \rd \xi \\
                - \langle \nabla \psi(t^\ast, x^\ast(\cdot)), z(\tau) - x^\ast(t^\ast) \rangle \biggr|
                \to 0
                \text{ as } \tau \downarrow t^\ast.
            \end{multline}
            Relations \eqref{estimate_Lebesgue}--\eqref{estimate_Lebesgue_second} imply (see \eqref{ci_derivatives}) that the functional $\psi$ is $ci$-differentiable at the point $(t^\ast, x^\ast(\cdot))$ and that the corresponding $ci$-derivatives are as in \eqref{psi_ci_derivatives_t} and \eqref{psi_ci_derivatives_nabla}.

            Finally, note that the first inequality in \eqref{lemma_variational_principle_G_part_ii} follows directly from \eqref{psi_ci_derivatives_t}.
            Furthermore, using \eqref{nabla_V_estimates} and recalling that $(t^\ast, x^\ast(\cdot)) \in X \subset G(\alpha)$ and $(\tau^{(k)}, y^{(k)}(\cdot)) \in X \subset G(\alpha)$ for all $k \in \mathbb{N} \cup \{0\}$, we get
            \begin{equation*}
                \|\nabla \psi (t^\ast, x^\ast(\cdot))\|
                \leq \varkappa \sum_{k = 0}^{\infty} \frac{1}{2^{k - 1}} \|x^\ast(t^\ast) - y^{(k)} (t^\ast \wedge \tau^{(k)})\|
                \leq 8 \alpha \varkappa,
            \end{equation*}
            which proves the second inequality in \eqref{lemma_variational_principle_G_part_ii}.
            The proof of Lemma \ref{lemma_variational_principle_G} is complete.

        \subsubsection{Comments}
        \label{subsubsection_remarks}

            In this section, we make some comments concerning Lemma \ref{lemma_variational_principle_G} and its proof.

            \begin{remark} \label{remarks_on_variational_principle}
            (i)
                Comparing Lemma \ref{lemma_variational_principle_G} with \cite[Lemma 2.3]{Zhou_2020_2} (see also, e.g., formula (4.5) in that paper), we note that, according to \eqref{lemma_variational_principle_G_part_iii}, the point $(t^\ast, x^\ast(\cdot))$ provides the minimum over the whole set $X$ and not only over its subset $\{(t, x(\cdot)) \in X \ | \ t \geq t^\ast\}$.
                This difference is important for the proof of Lemma \ref{lemma_CL} given below (see, e.g., relation \eqref{gamma_k_above}).

            \smallskip

            \noindent (ii)
                The fact that we define the functional $\mu_\alpha$ equal to $c_\alpha$ for $t < \tau$ (see \eqref{varrho_alpha}) is used to obtain inequalities \eqref{t_ast_geq_t_i}, which allow us to get $ci$-differentiability of the functional $\psi$ at the point $(t^\ast, x^\ast(\cdot))$ directly by $ci$-smoothness of the functional $V$.

            \smallskip

            \noindent (iii)
                Let us consider the functional $\bar{V} \colon G \times G \to \mathbb{R}$ given by (see \eqref{V})
                \begin{equation} \label{widetilde_V}
                    \bar{V} \bigl( (t, x(\cdot)), (\tau, y(\cdot)) \bigr)
                    \triangleq V(T, x(\cdot \wedge t) - y(\cdot \wedge \tau)),
                    \quad \forall (t, x(\cdot)), (\tau, y(\cdot)) \in G.
                \end{equation}
                In explicit form, we have
                \begin{multline*}
                    \bar{V} \bigl( (t, x(\cdot)), (\tau, y(\cdot)) \bigr) \\
                    = \frac{\bigl( \|x(\cdot \wedge t) - y(\cdot \wedge \tau)\|_{[- h, T]}^2
                    - \|x(t) - y(\tau)\|^2 \bigr)^2}{\|x(\cdot \wedge t) - y(\cdot \wedge \tau)\|_{[- h, T]}^2}
                    + \|x(t) - y(\tau)\|^2,
                \end{multline*}
                if $\|x(\cdot \wedge t) - y(\cdot \wedge \tau)\|_{[- h, T]} > 0$, and, otherwise, $\bar{V} ((t, x(\cdot)), (\tau, y(\cdot))) = 0$.
                Then, instead of the functional $\mu_\alpha$ from \eqref{varrho_alpha}, we could take the functional $\bar{\mu} \colon G \times G \to \mathbb{R}$ defined by
                \begin{equation*}
                    \bar{\mu} \bigl( (t, x(\cdot)), (\tau, y(\cdot)) \bigr)
                    \triangleq (t - \tau)^2 + \bar{V} \bigl( (t, x(\cdot)), (\tau, y(\cdot)) \bigr),
                    \quad \forall (t, x(\cdot)), (\tau, y(\cdot)) \in G.
                \end{equation*}
                We note that
                \begin{equation*}
                    \bar{V} \bigl( (t, x(\cdot)), (\tau, y(\cdot)) \bigr)
                    = V(t, x(\cdot) - y_t(\cdot \wedge \tau))
                \end{equation*}
                for all $(t, x(\cdot))$, $(\tau, y(\cdot)) \in G$ with $t \geq \tau$, and, thus, the functionals $\bar{\mu}$ and $\mu_\alpha$ differ only when $t < \tau$.
                However, if the functional $\bar{\mu}$ is taken, inequalities \eqref{t_ast_geq_t_i} may no longer hold, i.e., it may happen that $t^\ast < \tau^{(k)}$ for some $k \in \mathbb{N} \cup \{0\}$.
                In turn, this may lead to the fact that the functional $\psi$ defined by \eqref{psi_definition} where we now replace $\mu_\alpha$ with $\bar{\mu}$ is not $ci$-differentiable at the point $(t^\ast, x^\ast(\cdot))$.
                More precisely, let us emphasize that, if we fix $(\tau_\ast, y_\ast(\cdot)) \in G$, then the functional
                \begin{equation} \label{widetilde_V^ast}
                    G \ni (t, x(\cdot)) \mapsto \bar{V}_\ast (t, x(\cdot))
                    \triangleq \bar{V} \bigl( (t, x(\cdot)), (\tau_\ast, y_\ast(\cdot)) \bigr) \in \mathbb{R}
                \end{equation}
                may not be $ci$-differentiable at some point $(t_\ast, x_\ast(\cdot)) \in G_0$ with $t_\ast < \tau_\ast$.
                To illustrate this circumstance, we present a simple example below.
            \end{remark}

            \begin{example} \label{example}
                Let us suppose that $n = 1$ and $T = 1$ for simplicity.
                Let us take
                \begin{equation*}
                    \tau_\ast
                    \triangleq 1,
                    \quad y_\ast(\xi)
                    \triangleq \begin{cases}
                        0, & \mbox{if } \xi \in [- h, 0], \\
                        \xi, & \mbox{if } \xi \in (0, 1],
                      \end{cases}
                \end{equation*}
                and show that the corresponding functional $\bar{V}_\ast$ (see \eqref{widetilde_V^ast}) is not $ci$-differentiable at the point $(t_\ast, x_\ast(\cdot)) \in G_0$ with
                \begin{equation*}
                    t_\ast
                    \triangleq 0,
                    \quad x_\ast(\xi)
                    \triangleq 1,
                    \quad \forall \xi \in [- h, 0].
                \end{equation*}
                Arguing by contradiction, let us assume that there exist $\partial_t \bar{V}_\ast(0, x_\ast(\cdot))$, $\nabla \bar{V}_\ast (0, x_\ast(\cdot)) \in \mathbb{R}$ such that, for all $z(\cdot) \in \Lip(0, x_\ast(\cdot))$,
                \begin{equation} \label{assumption_ci_differentiability}
                    \lim_{\tau \downarrow 0} \frac{\bar{V}_\ast(\tau, z_\tau(\cdot)) - \bar{V}_\ast(0, x_\ast(\cdot))
                    - \partial_t \bar{V}_\ast(0, x_\ast(\cdot)) \tau - \nabla \bar{V}_\ast (0, x_\ast(\cdot)) (z(\tau) - 1)}{\tau}
                    = 0.
                \end{equation}
                First of all, let us note that $\bar{V}_\ast (0, x_\ast(\cdot)) = 1$.
                Further, in accordance with \eqref{z^l}, let us consider the function $z^{[0]}(\cdot) \in \Lip(0, x_\ast(\cdot))$ such that $z^{[0]}(\tau) \triangleq 1$ for all $\tau \in [0, 1]$.
                Then, we have $\bar{V}_\ast (\tau, z^{[0]}_\tau(\cdot)) = 1$ for all $\tau \in [0, 1]$, wherefrom, owing to \eqref{assumption_ci_differentiability}, we find that
                \begin{equation} \label{assumption_ci_differentiability_corollary}
                    \partial_t \bar{V}_\ast(0, x_\ast(\cdot))
                    = 0.
                \end{equation}
                Now, let us fix $l > 1$ and take the function $z^{[l]}(\cdot) \in \Lip(0, x_\ast(\cdot))$ such that $z^{[l]}(\tau) \triangleq 1 + \tau l$ for all $\tau \in [0, 1]$.
                For every $\tau \in [0, 1]$, since
                \begin{equation*}
                    \max_{\xi \in [- h, 1]} |z^{[l]}(\xi \wedge \tau) - y_\ast(\xi)|
                    = 1 + \tau (l - 1),
                    \quad |z^{[l]}(\tau) - y_\ast(1)| = \tau l,
                \end{equation*}
                we get
                \begin{equation*}
                    \bar{V}_\ast (\tau, z^{[l]}_\tau(\cdot))
                    = \frac{((1 + \tau (l - 1))^2 - \tau^2 l^2)^2}{(1 + \tau (l - 1))^2} + \tau^2 l^2.
                \end{equation*}
                Hence, it follows from \eqref{assumption_ci_differentiability} and \eqref{assumption_ci_differentiability_corollary} that
                \begin{equation*}
                    \lim_{\tau \downarrow 0} \frac{1}{\tau}
                    \biggl( \frac{((1 + \tau (l - 1))^2 - \tau^2 l^2)^2}{(1 + \tau (l - 1))^2} + \tau^2 l^2 - 1 \biggr)
                    = \nabla \bar{V}_\ast (0, x_\ast(\cdot)) l.
                \end{equation*}
                After calculating the limit, we obtain $\nabla \bar{V}_\ast (0, x_\ast(\cdot)) =2 (l - 1) / l$.
                However, by definition, the $ci$-derivative $\nabla \bar{V}_\ast (0, x_\ast(\cdot))$ cannot depend on the choice of $l$, and, thus, we come to a contradiction.
            \end{example}

    \subsection{Property of $ci$-subdifferential}
    \label{subsection_Clarke_Ledyaev}

        Let $\varphi \colon G \to \mathbb{R}$, let $(t, x(\cdot)) \in G_0$, and let $L \subset \mathbb{R}^n$ be a non-empty convex compact set.
        Let us introduce the value (in this regard, see also, e.g., \cite[Section 4]{Clarke_Ledyaev_1994})
        \begin{multline} \label{new_lower_directional_derivative_multivalued}
            \rd^-_0 \varphi(t, x(\cdot); L) \\
            \triangleq \lim_{\delta \downarrow 0} \inf
            \biggl\{ \frac{\varphi(\tau, \omega_\tau(\cdot)) - \varphi(t, x(\cdot))}{\tau - t} \biggm|
            \, \tau \in (t, t + \delta], \, \omega(\cdot) \in \Omega(t, x(\cdot), [L]^\delta) \biggr\},
        \end{multline}
        where the set $\Omega(t, x(\cdot), [L]^\delta)$ is given by \eqref{Omega}.

        \begin{remark}
            Along with the value $\rd^- \varphi(t, x(\cdot); L)$ from \eqref{derivatives_multi-valued}, the value $\rd^-_0 \varphi(t, x(\cdot); L)$ can also be interpreted as a lower right derivative of the functional $\varphi$ at the point $(t, x(\cdot))$ in the multi-valued direction $L$.
            Moreover, note that the following inequality holds:
            \begin{equation} \label{new_and_old_lower_directional_derivatives_multivalued}
                \rd^-_0 \varphi(t, x(\cdot); L)
                \leq \rd^- \varphi(t, x(\cdot); L).
            \end{equation}
            On the other hand, the distinguishing feature of the value $\rd^-_0 \varphi(t, x(\cdot); L)$ is that it is not expressed directly in terms of right derivatives along extensions (see \eqref{derivatives_along_extenstion}), which means that it does not quite agree with the notion of $ci$-differentiability (see \eqref{ci_derivatives}).
            Nevertheless, we need to introduce the value $\rd^-_0 \varphi(t, x(\cdot); L)$ as an intermediate auxiliary technical construction.
        \end{remark}

        For every $(t, x(\cdot)) \in G$ and every $\eta > 0$, let us denote
        \begin{equation} \label{neighbourhood}
            O_\eta(t, x(\cdot))
            \triangleq \bigl\{ (\tau, y(\cdot)) \in G \bigm|
            \rho_\infty \bigl( (t, x(\cdot)), (\tau, y(\cdot)) \bigr)
            \leq \eta \bigr\}.
        \end{equation}

        In this section, we prove
        \begin{lemma} \label{lemma_CL}
            Let $\varphi \in \Phi_+$, let $(t_\ast, x_\ast(\cdot)) \in G_0$, and let $L \subset \mathbb{R}^n$ be a non-empty convex compact set.
            Suppose that
            \begin{equation} \label{lemma_CL_inequality_for_derivatives}
                \rd^-_0 \varphi(t_\ast, x_\ast(\cdot); L)
                > 0.
            \end{equation}
            Then, for every $\eta > 0$, there exist $(t, x(\cdot)) \in O_\eta(t_\ast, x_\ast(\cdot))$ and $(p_0, p) \in D^-\varphi (t, x(\cdot))$ such that
            \begin{equation} \label{lemma_CL_main}
                p_0 + \langle l, p \rangle
                > 0,
                \quad \forall l \in L.
            \end{equation}
        \end{lemma}
        \begin{proof}
            Since $(t_\ast, x_\ast(\cdot)) \in G_0$, we can assume that $O_\eta(t_\ast, x_\ast(\cdot)) \subset G_0$.
            Moreover, due to the inclusion $\varphi \in \Phi_+$, the functional $\varphi$ is lower semi-continuous at the point $(t_\ast, x_\ast(\cdot))$, and, hence, we can also assume that
            \begin{equation} \label{choice_eta}
                \varphi(t, x(\cdot))
                \geq \varphi(t_\ast, x_\ast(\cdot)) - 1,
                \quad \forall (t, x(\cdot)) \in O_\eta(t_\ast, x_\ast(\cdot)).
            \end{equation}
            Note that $O_\eta(t_\ast, x_\ast(\cdot)) \subset G(\alpha)$ with $\alpha \triangleq \|x_\ast(\cdot)\|_{[- h, t_\ast]} + \eta > 0$.
            Owing to \eqref{lemma_CL_inequality_for_derivatives}, and in accordance with \eqref{new_lower_directional_derivative_multivalued}, we can choose $\delta_\ast \in (0, T - t_\ast]$ and $\varepsilon_\ast > 0$ such that
            \begin{equation} \label{lemKL_cond}
                \varphi(\tau, \omega_\tau(\cdot)) - \varphi(t_\ast, x_\ast(\cdot))
                \geq 2 \varepsilon_\ast (\tau - t_\ast),
                \quad \forall \tau \in (t_\ast, t_\ast + \delta_\ast],
                \quad \forall \omega(\cdot) \in \Omega,
            \end{equation}
            where we denote $\Omega \triangleq \Omega(t_\ast, x_\ast(\cdot), L)$ (see~\eqref{Omega}).
            Let us take $\lambda_L > \max_{l \in L} \|l\|$ and fix $\delta > 0$ such that
            \begin{equation} \label{choice_delta}
                \delta
                \leq \delta_\ast,
                \quad \delta
                \leq \eta / (2 + \lambda_L).
            \end{equation}

            Let us consider the sets
            \begin{multline} \label{R}
                X
                \triangleq \bigl\{ (t, x(\cdot)) \in G \bigm|
                t \in [t_\ast, t_\ast + \delta], \,
                x_{t_\ast}(\cdot) = x_\ast(\cdot), \\
                \text{and } \forall \xi \in [t_\ast, t] \ \exists \omega(\cdot) \in \Omega \colon \|x(\xi) - \omega(\xi)\| \leq \delta \bigr\}
            \end{multline}
            and
            \begin{equation} \label{D}
                Y
                \triangleq \bigl\{ (\tau, y(\cdot)) \in G \bigm|
                \tau \in [t_\ast, t_\ast + \delta] \text{ and } \exists \omega(\cdot) \in \Omega \colon y(\cdot) = \omega_\tau(\cdot) \bigr\}.
            \end{equation}
            Note that $(t_\ast, x_\ast(\cdot)) \in Y \subset X$.
            Further, using compactness of the set $\Omega$ in $C([- h, T], \mathbb{R}^n)$, it can be shown that the set $X$ is closed and the set $Y$ is compact.
            In addition, by the definition of $\lambda_L$, we obtain (see \eqref{Omega})
            \begin{equation} \label{choice_lambda_L}
                \|\omega(\tau) - \omega(\xi)\|
                \leq \lambda_L |\tau - \xi|,
                \quad \forall \tau, \xi \in [t_\ast, T],
                \quad \forall \omega(\cdot) \in \Omega,
            \end{equation}
            wherefrom, and in view of the choice of $\delta$ (see the second inequality in \eqref{choice_delta}), we derive the inclusion $X \subset O_\eta(t_\ast, x_\ast(\cdot))$.
            In particular, we get $X \subset G(\alpha) \cap G_0$.

            Let us introduce the functional $\Psi \colon G \times G \to \mathbb{R}$ defined by
            \begin{equation} \label{Psi}
                \Psi\bigl( (t, x(\cdot)), (\tau, y(\cdot)) \bigr)
                \triangleq \|x(t) - y(\tau)\|^2 + \int_{- h}^{T} \|x(\xi \wedge t) - y(\xi \wedge \tau)\|^2 \, \rd \xi
            \end{equation}
            for all $(t, x(\cdot))$, $(\tau, y(\cdot)) \in G$.
            Note that (see also, e.g., \cite[formula (5.13)]{Lukoyanov_2007_IMM_Eng}), for every fixed $(\tau_\ast, y_\ast(\cdot)) \in G$, the functional
            \begin{equation*}
                G \ni (t, x(\cdot)) \mapsto \Psi_\ast (t, x(\cdot)) \triangleq \Psi\bigl( (t, x(\cdot)), (\tau_\ast, y_\ast(\cdot)) \bigr) \in \mathbb{R}
            \end{equation*}
            is $ci$-smooth, and its $ci$-derivatives are as follows:
            \begin{equation} \label{Psi_ast_ci_derivatives}
                \partial_t \Psi_\ast (t, x(\cdot))
                = 0,
                \quad \nabla \Psi_\ast (t, x(\cdot))
                = 2 (x(t) - y_\ast(\tau_\ast))
                + 2 \int_{t}^{T} (x(t) - y_\ast(\xi \wedge \tau_\ast)) \, \rd \xi
            \end{equation}
            for all $(t, x(\cdot)) \in G_0$.
            Moreover, since
            \begin{equation} \label{Psi_ast_symmetry}
                \Psi\bigl( (t, x(\cdot)), (\tau, y(\cdot)) \bigr)
                = \Psi\bigl( (\tau, y(\cdot)), (t, x(\cdot)) \bigr),
                \quad \forall (t, x(\cdot)), (\tau, y(\cdot)) \in G,
            \end{equation}
            similar properties of the functional $\Psi$ are valid when the first argument is fixed.

            For every $k \in \mathbb{N}$, let us consider the functional $\zeta_k \colon X \times Y \to \mathbb{R}$ given by
            \begin{equation*}
                \zeta_k \bigl( (t, x(\cdot)), (\tau, y(\cdot)) \bigr)
                \triangleq \varphi(t, x(\cdot)) + k \Psi \bigl( (t, x(\cdot)), (\tau, y(\cdot)) \bigr) + k^4 (t - \tau)^2 - \varepsilon_\ast (\tau - t_\ast)
            \end{equation*}
            for all $(t, x(\cdot)) \in X$ and all $(\tau, y(\cdot)) \in Y$.
            Taking into account that the functional $\Psi$ is continuous and the set $Y$ is compact, let us define the functional $\varphi_k \colon X \to \mathbb{R}$ by
            \begin{multline} \label{f_k_definition}
                \varphi_k(t, x(\cdot))
                \triangleq \min_{(\tau, y(\cdot)) \in Y} \zeta_k \bigl( (t, x(\cdot)), (\tau, y(\cdot)) \bigr) \\
                = \varphi(t, x(\cdot)) + \min_{(\tau, y(\cdot)) \in Y}
                \bigl( k \Psi \bigl( (t, x(\cdot)), (\tau, y(\cdot)) \bigr) + k^4 (t - \tau)^2 - \varepsilon_\ast (\tau - t_\ast) \bigr)
            \end{multline}
            for all $(t, x(\cdot)) \in X$.
            Note that the functional $\varphi_k$ is lower semi-continuous and bounded from below.
            Indeed, $\varphi$ is lower semi-continuous on $G$ by the inclusion $\varphi \in \Phi_+$ and is bounded below on $X$ owing to the inclusion $X \subset O_\eta(t_\ast, x_\ast(\cdot))$ and inequality \eqref{choice_eta}, while the functional
            \begin{equation} \label{functional_auxiliary}
                X \ni (t, x(\cdot))
                \mapsto \min_{(\tau, y(\cdot)) \in Y}
                \bigl( k \Psi \bigl( (t, x(\cdot)), (\tau, y(\cdot)) \bigr) + k^4 (t - \tau)^2 - \varepsilon_\ast (\tau - t_\ast) \bigr)
                \in \mathbb{R}
            \end{equation}
            is continuous (see, e.g., \cite[Chapter 3, Section 1, Proposition 23]{Aubin_Ekeland_1984}) and, in view of non-negativeness of $\Psi$, is bounded from below.

            For every $k \in \mathbb{N}$, let us apply Lemma \ref{lemma_variational_principle_G} for the functional $\varphi_k \colon X \to \mathbb{R}$ and the number $\varkappa_k \triangleq 1 / (2 k)$.
            As a result, we obtain that, for every $k \in \mathbb{N}$, there exist a point $(t^{(k)}, x^{(k)}(\cdot)) \in X$ and a functional $\psi_k \colon G \to \mathbb{R}$ such that the following statements hold:

            \smallskip

            \noindent (i)
                The estimate below is fulfilled:
                \begin{equation} \label{alpha_psi}
                    |\psi_k(t, x(\cdot))|
                    \leq c_\alpha / k,
                    \quad \forall (t, x(\cdot)) \in X, \quad \forall k \in \mathbb{N}.
                \end{equation}

            \smallskip

            \noindent (ii)
                For every $k \in \mathbb{N}$, the functional $\psi_k$ is $ci$-differentiable at the point $(t^{(k)}, x^{(k)}(\cdot))$, and the following limit relations are valid as $k \to \infty$:
                \begin{equation} \label{ci_derivatives_psi_k}
                    | \partial_t \psi_k(t^{(k)}, x^{(k)}(\cdot)) |
                    \to 0,
                    \quad \| \nabla \psi_k(t^{(k)}, x^{(k)}(\cdot)) \|
                    \to 0.
                \end{equation}

            \smallskip

            \noindent (iii)
                It holds that
                \begin{equation} \label{choice_t^k_x^k}
                    \varphi_k (t^{(k)}, x^{(k)}(\cdot)) + \psi_k (t^{(k)}, x^{(k)}(\cdot))
                    = \min_{(t, x(\cdot)) \in X} \bigl( \varphi_k (t, x(\cdot)) + \psi_k (t, x(\cdot)) \bigr),
                    \quad \forall k \in \mathbb{N}.
                \end{equation}

            \smallskip

            Now, for every $k \in \mathbb{N}$, let us consider the functional $\gamma_k \colon X \times Y \to \mathbb{R}$ given by
            \begin{multline*}
                \gamma_k \bigl( (t, x(\cdot)), (\tau, y(\cdot)) \bigr)
                \triangleq \zeta_k \bigl( (t, x(\cdot)), (\tau, y(\cdot)) \bigr) + \psi_k(t, x(\cdot)) \\
                = \varphi(t, x(\cdot)) + k \Psi \bigl( (t, x(\cdot)), (\tau, y(\cdot)) \bigr) + k^4 (t - \tau)^2 - \varepsilon_\ast (\tau - t_\ast)
                + \psi_k(t, x(\cdot))
            \end{multline*}
            for all $(t, x(\cdot)) \in X$ and all $(\tau, y(\cdot)) \in Y$.
            According to definition \eqref{f_k_definition} of the functional $\varphi_k$, let us choose $(\tau^{(k)}, y^{(k)}(\cdot)) \in Y$ from the condition
            \begin{equation*}
                \zeta_k \bigl( (t^{(k)}, x^{(k)}(\cdot)), (\tau^{(k)}, y^{(k)}(\cdot)) \bigr)
                = \varphi_k (t^{(k)}, x^{(k)}(\cdot)).
            \end{equation*}
            Then, and due to \eqref{choice_t^k_x^k}, we get
            \begin{equation}\label{min}
                \gamma_k \bigl( (t^{(k)}, x^{(k)}(\cdot)), (\tau^{(k)}, y^{(k)}(\cdot)) \bigr)
                = \min_{(t, x(\cdot)) \in X, \, (\tau, y(\cdot)) \in Y}
                \gamma_k \bigl( (t, x(\cdot)), (\tau, y(\cdot)) \bigr).
            \end{equation}

            From \eqref{min}, and using the equality $\Psi((t_\ast, x_\ast(\cdot)), (t_\ast, x_\ast(\cdot))) = 0$, we derive
            \begin{align}
                \gamma_k \bigl( (t^{(k)}, x^{(k)}(\cdot)), (\tau^{(k)}, y^{(k)}(\cdot)) \bigr)
                & \leq \gamma_k \bigl( (t_\ast, x_\ast(\cdot)), (t_\ast, x_\ast(\cdot)) \bigr) \nonumber \\
                & = \varphi(t_\ast, x_\ast(\cdot)) + \psi_k(t_\ast, x_\ast(\cdot)),
                \quad \forall k \in \mathbb{N}.
                \label{gamma_k_above}
            \end{align}
            Thus, due to inequalities \eqref{choice_eta} and \eqref{alpha_psi} and non-negativeness of $\Psi$, we obtain
            \begin{equation} \label{beta}
                \Psi \bigl( (t^{(k)}, x^{(k)}(\cdot)), (\tau^{(k)}, y^{(k)}(\cdot)) \bigr)
                \leq \beta / k,
                \quad (t^{(k)} - \tau^{(k)})^2
                \leq \beta / k^4,
                \quad \forall k \in \mathbb{N},
            \end{equation}
            with $\beta \triangleq 1 + 2 c_\alpha + \varepsilon_\ast \delta > 0$.
            In particular, we get $\Psi((t^{(k)}, x^{(k)}(\cdot)), (\tau^{(k)}, y^{(k)}(\cdot))) \to 0$ and $|t^{(k)} - \tau^{(k)}| \to 0$ as $k \to \infty$.
            Hence, and since
            \begin{multline*}
                \rho_1 \bigl( (t^{(k)}, x^{(k)}(\cdot)), (\tau^{(k)}, y^{(k)}(\cdot)) \bigr) \\
                \leq |t^{(k)} - \tau^{(k)}|
                + (1 + \sqrt{T + h}) \sqrt{\Psi \bigl( (t^{(k)}, x^{(k)}(\cdot)), (\tau^{(k)}, y^{(k)}(\cdot)) \bigr)},
                \quad \forall k \in \mathbb{N},
            \end{multline*}
            in view of definitions \eqref{rho_1} and \eqref{Psi} of $\rho_1$ and $\Psi$, we find that
            \begin{equation} \label{lim0}
                \rho_1 \bigl( (t^{(k)}, x^{(k)}(\cdot)), (\tau^{(k)}, y^{(k)}(\cdot)) \bigr)
                \to 0
                \text{ as } k \to \infty.
            \end{equation}
            Further, by compactness of $Y$, we can assume that there exists $(\bar{\tau}, \bar{y}(\cdot)) \in Y$ such that
            \begin{equation} \label{lim1}
                \rho_\infty \bigl( (\tau^{(k)}, y^{(k)}(\cdot)), (\bar{\tau}, \bar{y}(\cdot)) \bigr)
                \to 0
                \text{ as } k \to \infty.
            \end{equation}
            Then, we have
            \begin{equation} \label{lim2}
                \tau^{(k)}
                \to \bar{\tau},
                \quad t^{(k)}
                \to \bar{\tau}
            \end{equation}
            as $k \to \infty$, and, owing to relationship \eqref{rho_1_and_rho_infty} between the metrics $\rho_1$ and $\rho_\infty$, we derive from \eqref{lim0} and \eqref{lim1} that
            \begin{equation} \label{lim2.5}
                \rho_1 \bigl( (t^{(k)}, x^{(k)}(\cdot)), (\bar{\tau}, \bar{y}(\cdot)) \bigr)
                \to 0
                \text{ as } k \to \infty.
            \end{equation}

            Let us show that
            \begin{equation} \label{overline_tau_leq}
                \bar{\tau}
                < t_\ast + \delta.
            \end{equation}
            Note that $\bar{\tau} \leq t_\ast + \delta$ according to definition \eqref{D} of the set $Y$.
            Consequently, arguing by contradiction, we can assume that $\bar{\tau} = t_\ast + \delta$.
            Due to \eqref{gamma_k_above}, and recalling that $\Psi$ is non-negative, we obtain
            \begin{equation} \label{contr}
                \varphi(t_\ast, x_\ast(\cdot)) + \psi_k(t_\ast, x_\ast(\cdot))
                \geq \varphi(t^{(k)}, x^{(k)}(\cdot)) - \varepsilon_\ast (\tau^{(k)} - t_\ast) + \psi_k(t^{(k)}, x^{(k)}(\cdot))
            \end{equation}
            for all $k \in \mathbb{N}$.
            Since $\varphi \in \Phi_+$ and $\{(t^{(k)}, x^{(k)}(\cdot))\}_{k = 1}^\infty \subset X \subset G(\alpha)$, relation \eqref{lim2.5} yields
            \begin{equation} \label{convergence_of_varphi}
                \liminf_{k \to \infty}
                \varphi(t^{(k)}, x^{(k)}(\cdot))
                \geq \varphi(\bar{\tau}, \bar{y}(\cdot)).
            \end{equation}
            Therefore, by \eqref{contr}, and taking \eqref{alpha_psi} and \eqref{lim2} into account, we get
            \begin{equation*}
                \varphi(t_\ast, x_\ast(\cdot))
                \geq \varphi(\bar{\tau}, \bar{y}(\cdot)) - \varepsilon_\ast (\bar{\tau} - t_\ast)
                = \varphi(t_\ast + \delta, \bar{y}(\cdot)) - \varepsilon_\ast \delta.
            \end{equation*}
            From the inclusion $(t_\ast + \delta, \bar{y}(\cdot)) \in Y$, it follows that there exists a function $\bar{\omega}(\cdot) \in \Omega$ such that $\bar{y}(\cdot) = \bar{\omega}_{t_\ast + \delta}(\cdot)$.
            Thus, we come to the inequality
            \begin{equation*}
                \varphi(t_\ast, x_\ast(\cdot))
                \geq \varphi(t_\ast + \delta, \bar{\omega}_{t_\ast + \delta}(\cdot)) - \varepsilon_\ast \delta,
            \end{equation*}
            which contradicts \eqref{lemKL_cond} in view of the choice of $\delta$ (see the first inequality in \eqref{choice_delta}).

            Let us fix $k \in \mathbb{N}$ such that (see \eqref{lim0}, \eqref{lim2}, and \eqref{overline_tau_leq})
            \begin{equation} \label{k}
                \begin{array}{c}
                    |t^{(k)} - \tau^{(k)}|
                    \leq \delta / (3 \lambda_L),
                    \quad \|x^{(k)}(t^{(k)}) - y^{(k)}(\tau^{(k)})\|
                    \leq \delta / 3, \\[0.5em]
                    t^{(k)} < t_\ast + \delta,
                    \quad \tau^{(k)} < t_\ast + \delta,
                \end{array}
            \end{equation}
            and (see \eqref{ci_derivatives_psi_k} and \eqref{beta})
            \begin{equation} \label{k_2}
                \begin{array}{c}
                    |\partial_t \psi_k (t^{(k)}, x^{(k)}(\cdot))|
                    \leq \varepsilon_\ast / 4,
                    \quad \|\nabla \psi_k (t^{(k)}, x^{(k)}(\cdot))\|
                    \leq \varepsilon_\ast / (4 \lambda_L), \\[0.5em]
                    |t^{(k)} - \tau^{(k)}|
                    \leq \varepsilon_\ast / (16 k \alpha \lambda_L).
                \end{array}
            \end{equation}
            Set
            \begin{equation} \label{p_0_p}
                p_0
                \triangleq - 2 k^4 (t^{(k)} - \tau^{(k)}) - \partial_t \psi_k (t^{(k)}, x^{(k)}(\cdot)),
                \quad p
                \triangleq - k \nabla_1 \Psi - \nabla \psi_k (t^{(k)}, x^{(k)}(\cdot)),
            \end{equation}
            where we denote (see \eqref{Psi_ast_ci_derivatives})
            \begin{equation} \label{nabla_1_Psi}
                \nabla_1 \Psi
                \triangleq 2 (x^{(k)}(t^{(k)}) - y^{(k)}(\tau^{(k)}))
                + 2 \int_{t^{(k)}}^{T} (x^{(k)}(t^{(k)}) - y^{(k)}(\xi \wedge \tau^{(k)})) \, \rd \xi.
            \end{equation}
            Let us show that the statement of the lemma is valid for the point $(t^{(k)}, x^{(k)}(\cdot))$ and the pair $(p_0, p)$.
            We first recall that $(t^{(k)}, x^{(k)}(\cdot)) \in X \subset O_\eta(t_\ast, x_\ast(\cdot))$.

            Let us prove the inclusion $(p_0, p) \in D^- \varphi(t^{(k)}, x^{(k)}(\cdot))$.
            To this end (see \eqref{subdifferential}), we need to take arbitrarily $z(\cdot) \in \Lip(t^{(k)}, x^{(k)}(\cdot))$ and show that
            \begin{equation} \label{proof_of_subdifferential}
                \liminf_{\tau \downarrow t^{(k)}} \frac{ \varphi(\tau, z_\tau(\cdot)) - \varphi(t^{(k)}, x^{(k)}(\cdot))
                - p_0 (\tau - t^{(k)}) - \langle p, z(\tau) - x^{(k)}(t^{(k)}) \rangle}{\tau - t^{(k)}}
                \geq 0.
            \end{equation}
            Let $\lambda_z > 0$ be such that
            \begin{equation} \label{choice_lambda_z}
                \|z(\tau) - z(\xi)\|
                \leq \lambda_z |\tau - \xi|,
                \quad \forall \tau, \xi \in [t^{(k)}, T].
            \end{equation}
            Let us choose $\delta_z > 0$ from the conditions (see also the third inequality in \eqref{k})
            \begin{equation} \label{choice_delta_z}
                t^{(k)} + \delta_z
                \leq t_\ast + \delta,
                \quad \delta_z
                \leq \delta / (3 (\lambda_z + \lambda_L)).
            \end{equation}
            Let us verify the inclusion $(t^{(k)} + \delta_z, z_{t^{(k)} + \delta_z}(\cdot)) \in X$.
            Since $z(\cdot) \in \Lip(t^{(k)}, x^{(k)}(\cdot))$ and $(t^{(k)}, x^{(k)}(\cdot)) \in X$, the equalities $z_{t_\ast}(\cdot) = x^{(k)}_{t_\ast}(\cdot) = x_\ast(\cdot)$ hold, and, for every $\xi \in [t_\ast, t^{(k)}]$, we can find $\omega_\ast(\cdot) \in \Omega$ such that
            \begin{equation*}
                \|z(\xi) - \omega_\ast(\xi)\|
                = \|x^{(k)}(\xi) - \omega_\ast(\xi)\|
                \leq \delta.
            \end{equation*}
            Further, by virtue of the inclusion $(\tau^{(k)}, y^{(k)}(\cdot)) \in Y$, there exists $\omega^\ast(\cdot) \in \Omega$ for which $y^{(k)}(\cdot) = \omega^\ast_{\tau^{(k)}}(\cdot)$.
            Then, using \eqref{choice_lambda_L} and \eqref{choice_lambda_z} as well as the first two inequalities in \eqref{k} and the second inequality in \eqref{choice_delta_z}, we derive
            \begin{align*}
                \|z(\xi) - \omega^\ast(\xi)\|
                & \leq \|z(\xi) - x^{(k)}(t^{(k)})\| + \|x^{(k)}(t^{(k)}) - y^{(k)}(\tau^{(k)})\| \\
                & \quad + \|\omega^\ast(\tau^{(k)}) - \omega^\ast(t^{(k)})\| + \|\omega^\ast(t^{(k)}) - \omega^\ast(\xi)\| \\
                & \leq \lambda_z (\xi - t^{(k)}) + \delta / 3 + \lambda_L |\tau^{(k)} - t^{(k)}| + \lambda_L |t^{(k)} - \xi| \\
                & \leq (\lambda_z + \lambda_L) \delta_z + 2 \delta / 3
                \leq \delta,
                \quad \forall \xi \in (t^{(k)}, t^{(k)} + \delta_z].
            \end{align*}
            Thus, we conclude that $(t^{(k)} + \delta_z, z_{t^{(k)} + \delta_z}(\cdot)) \in X$.
            As a consequence, we get $(\tau, z_\tau(\cdot)) \in X$ for all $\tau \in [t^{(k)}, t^{(k)} + \delta_z]$.
            Hence, according to \eqref{min}, we obtain
            \begin{align*}
                0
                & \leq \gamma_k \bigl( (\tau, z_\tau(\cdot)),(\tau^{(k)}, y^{(k)}(\cdot)) \bigr)
                - \gamma_k \bigl( (t^{(k)}, x^{(k)}(\cdot)), (\tau^{(k)}, y^{(k)}(\cdot)) \bigr) \\
                & = \varphi(\tau, z_\tau(\cdot)) - \varphi(t^{(k)}, x^{(k)}(\cdot)) \\
                & \quad + k \Psi \bigl( (\tau, z_\tau(\cdot)), (\tau^{(k)}, y^{(k)}(\cdot)) \bigr)
                - k \Psi \bigl( (t^{(k)}, x^{(k)}(\cdot)), (\tau^{(k)}, y^{(k)}(\cdot)) \bigr) \\
                & \quad + k^4 (\tau - \tau^{(k)})^2 - k^4 (t^{(k)} - \tau^{(k)})^2
                + \psi_k(\tau, z_\tau(\cdot)) - \psi_k(t^{(k)}, x^{(k)}(\cdot))
            \end{align*}
            for all $\tau \in [t^{(k)}, t^{(k)} + \delta_z]$.
            From these relations, noting that (see \eqref{Psi_ast_ci_derivatives} and \eqref{nabla_1_Psi})
            \begin{multline} \label{nabla_1_Psi_property}
                \lim_{\tau \downarrow t^{(k)}}
                \frac{1}{\tau - t^{(k)}} \Bigl( \Psi \bigl( (\tau, z_\tau(\cdot)), (\tau^{(k)}, y^{(k)}(\cdot)) \bigr)
                - \Psi \bigl( (t^{(k)}, x^{(k)}(\cdot)), (\tau^{(k)}, y^{(k)}(\cdot)) \bigr) \\
                - \langle \nabla_1 \Psi, z(\tau) - x^{(k)}(t^{(k)}) \rangle \Bigr)
                = 0,
            \end{multline}
            recalling that $\psi_k$ is $ci$-differentiable at the point $(t^{(k)}, x^{(k)}(\cdot))$, and taking the definitions of $p_0$ and $p$ into account (see \eqref{p_0_p}), we derive inequality \eqref{proof_of_subdifferential}.

            Finally, let us prove inequality \eqref{lemma_CL_main}.
            Let us fix $l \in L$ and, in accordance with \eqref{z^l}, consider the function $\omega^{[l]}(\cdot) \in \Lip(\tau^{(k)}, y^{(k)}(\cdot))$ such that $\omega^{[l]}(\xi) \triangleq y^{(k)}(\tau^{(k)}) + (\xi - \tau^{(k)}) l$ for all $\xi \in [\tau^{(k)}, T]$.
            Due to the inclusions $(\tau^{(k)}, y^{(k)}(\cdot)) \in Y$ and $l \in L$, we have $\omega^{[l]}(\cdot) \in \Omega$.
            Hence, if we choose $\delta_l > 0$ from the condition $\tau^{(k)} + \delta_l \leq t_\ast + \delta$ (see the last inequality in \eqref{k}), we get $(\tau, \omega^{[l]}_\tau(\cdot)) \in Y$ for all $\tau \in [\tau^{(k)}, \tau^{(k)} + \delta_l]$.
            Then, and by \eqref{min}, we get
            \begin{align}
                0
                & \leq \gamma_k \bigl( (t^{(k)}, x^{(k)}(\cdot)), (\tau, \omega^{[l]}_\tau(\cdot)) \bigr)
                - \gamma_k \bigl( (t^{(k)}, x^{(k)}(\cdot)), (\tau^{(k)}, y^{(k)}(\cdot)) \bigr) \nonumber \\
                & = k \Psi \bigl( (t^{(k)}, x^{(k)}(\cdot)), (\tau, \omega^{[l]}_\tau(\cdot)) \bigr)
                - k \Psi \bigl( (t^{(k)}, x^{(k)}(\cdot)), (\tau^{(k)}, y^{(k)}(\cdot)) \bigr) \nonumber \\
                & \quad + k^4 (t^{(k)} - \tau)^2 - k^4 (t^{(k)} - \tau^{(k)})^2 - \varepsilon_\ast (\tau - \tau^{(k)})
                \label{proof_second_part}
            \end{align}
            for all $\tau \in [\tau^{(k)}, \tau^{(k)} + \delta_l]$.
            Since (see \eqref{Psi_ast_ci_derivatives} and \eqref{Psi_ast_symmetry})
            \begin{equation} \label{nabla_2_Psi_property}
                \lim_{\tau \downarrow \tau^{(k)}}
                \frac{\Psi \bigl( (t^{(k)}, x^{(k)}(\cdot)), (\tau, \omega^{[l]}_\tau(\cdot)) \bigr)
                - \Psi \bigl( (t^{(k)}, x^{(k)}(\cdot)), (\tau^{(k)}, y^{(k)}(\cdot)) \bigr)}{\tau - \tau^{(k)}}
                = \langle \nabla_2 \Psi, l \rangle
            \end{equation}
            with
            \begin{equation} \label{nabla_2_Psi}
                \nabla_2 \Psi
                \triangleq 2 (y^{(k)}(\tau^{(k)}) - x^{(k)}(t^{(k)}))
                + 2 \int_{\tau^{(k)}}^{T} (y^{(k)}(\tau^{(k)}) - x^{(k)}(\xi \wedge t^{(k)})) \, \rd \xi,
            \end{equation}
            it follows from \eqref{proof_second_part} that
            \begin{equation*}
                \varepsilon_\ast
                \leq k \langle \nabla_2 \Psi, l \rangle - 2 k^4 (t^{(k)} - \tau^{(k)}).
            \end{equation*}
            Hence, in accordance with the definitions of $p_0$ and $p$ (see \eqref{p_0_p}), we obtain
            \begin{equation} \label{last_part_1}
                p_0 + \langle p, l \rangle
                \geq \varepsilon_\ast - \partial_t \psi_k(t^{(k)}, x^{(k)}(\cdot))
                - k \langle \nabla_2 \Psi + \nabla_1 \Psi, l \rangle
                - \langle \nabla \psi_k (t^{(k)}, x^{(k)}(\cdot)), l \rangle.
            \end{equation}
            Recalling that $(t^{(k)}, x^{(k)}(\cdot))$, $(\tau^{(k)}, y^{(k)}(\cdot)) \in X \subset G(\alpha)$ and taking \eqref{nabla_1_Psi} and \eqref{nabla_2_Psi} into account, we derive
            \begin{equation} \label{last_part_2}
                \|\nabla_2 \Psi + \nabla_1 \Psi\|
                \leq 4 \alpha |t^{(k)} - \tau^{(k)}|.
            \end{equation}
            By \eqref{last_part_1} and \eqref{last_part_2}, and owing to \eqref{k_2}, we have $p_0 + \langle p, l \rangle \geq \varepsilon_\ast / 4$, which yields \eqref{lemma_CL_main}.
            The lemma is proved.
        \end{proof}

        Below, we make some comments on Lemma \ref{lemma_CL} and its proof.

        \begin{remark}
            (i)
                The set $X$ from \eqref{R} is not compact with respect to the auxiliary metric $\rho_1$ and, therefore, in view of \eqref{rho_1_and_rho_infty}, with respect to the metric $\rho_\infty$ too.
                Indeed, let us fix $e \in \mathbb{R}^n$ with $\|e\| = 1$ and $\omega(\cdot) \in \Omega$ and, for every $k \in \mathbb{N}$, consider the function $\tilde{x}^{(k)}(\cdot) \in C([- h, t_\ast + \delta], \mathbb{R}^n)$ given by
                \begin{equation*}
                    \tilde{x}^{(k)}(\xi)
                    = \begin{cases}
                        \omega(\xi),
                        & \mbox{if } \xi \in [- h, t_\ast], \\
                        \omega(\xi) + k (\xi - t_\ast) e,
                        & \mbox{if } \xi \in (t_\ast, t_\ast + \delta / k], \\
                        \omega(\xi) + \delta e,
                        & \mbox{if } \xi \in (t_\ast + \delta / k, t_\ast + \delta].
                    \end{cases}
                \end{equation*}
                Then, we have $\{(t_\ast + \delta, \tilde{x}^{(k)}(\cdot))\}_{k = 1}^\infty \subset X$.
                However, the sequence $\{(t_\ast + \delta, \tilde{x}^{(k)}(\cdot))\}_{k = 1}^\infty$ does not have a subsequence converging to some point $(t_\ast + \delta, \tilde{x}(\cdot)) \in G$ with respect to $\rho_1$, since, roughly speaking, this sequence actually converges with respect to $\rho_1$ but to the point $(t_\ast + \delta, \tilde{x}^\ast(\cdot))$ with the discontinuous function $\tilde{x}^\ast \colon [- h, t_\ast + \delta] \to \mathbb{R}$ defined by
                \begin{equation*}
                    \tilde{x}^\ast(\xi)
                    = \begin{cases}
                        \omega(\xi),
                        & \mbox{if } \xi \in [- h, t_\ast], \\
                        \omega(\xi) + \delta e,
                        & \mbox{if } \xi \in (t_\ast, t_\ast + \delta].
                    \end{cases}
                \end{equation*}
                As a result, we conclude that the existence of a minimum of the functional $\varphi_k$ from \eqref{f_k_definition} on $X$ does not follow from lower semi-continuity of $\varphi_k$ or even from the inclusion $\varphi_k \in \Phi_+$, which is valid since functional \eqref{functional_auxiliary} in fact belongs to $\Phi_+$.
                To overcome this difficulty, Lemma \ref{lemma_variational_principle_G} is applied.

            \smallskip

            \noindent (ii)
                The only place in the proof where we use the assumption $\varphi \in \Phi_+$ and not just the fact that $\varphi$ is lower semi-continuous is the proof of inequality \eqref{overline_tau_leq}, which in turn is based on relation \eqref{convergence_of_varphi}.

            \smallskip

            \noindent (iii)
                The choice of the functional $\Psi$ in form \eqref{Psi} leads us to the fact that we derive convergence \eqref{lim0} with respect to the auxiliary metric $\rho_1$.
                In order to obtain such a convergence but with respect to the metric $\rho_\infty$, which would allow us to replace the assumption $\varphi \in \Phi_+$ with the requirement that $\varphi$ is lower semi-continuous, we could try to take the functional $\bar{V}$ from \eqref{widetilde_V} instead of $\Psi$ and use the first inequality in \eqref{V_estimates}.
                However, in this case, there would be difficulties with the $ci$-differentiability properties \eqref{nabla_1_Psi_property} and \eqref{nabla_2_Psi_property}.
                Indeed, as shown in Example \ref{example} (see also item (iii) of Remark \ref{remarks_on_variational_principle}), one of these properties may fail to be valid if $t^{(k)} \neq \tau^{(k)}$.

            \smallskip

            \noindent (iv)
                Another way of how one could try to drop the assumption $\varphi \in \Phi_+$ but still obtain inequality \eqref{overline_tau_leq} is to suppose that $\varphi \colon G \to \mathbb{R}$ is continuous and replace the set $X$ from \eqref{R} with the set
                \begin{multline*}
                    X_\ast
                    \triangleq \bigl\{ (t, x(\cdot)) \in G \bigm|
                    t \in [t_\ast, t_\ast + \delta], \,
                    x_{t_\ast}(\cdot) = x_\ast(\cdot), \\
                    \text{and } \exists \omega(\cdot) \in \Omega \colon \|x(\xi) - \omega(\xi)\| \leq \theta, \, \forall \xi \in [t_\ast, t] \bigr\},
                \end{multline*}
                where, in view of compactness of the set $\Omega$ and continuity of $\varphi$, we choose $\theta \in (0, \delta]$ such that the following property holds: for every $(t, x(\cdot)) \in X_\ast$, there exists $(t, \tilde{y}(\cdot)) \in Y$ satisfying the condition
                \begin{equation*}
                    \varphi(t, x(\cdot))
                    \geq \varphi(t, \tilde{y}(\cdot)) - \varepsilon_\ast \delta / 2.
                \end{equation*}
                Note that the set $X_\ast$ is closed and the inclusions $Y \subset X_\ast \subset X$ are valid.
                Then, arguing similarly to the given proof, we arrive at inequality \eqref{contr}.
                However, now, for any $k \in \mathbb{N}$, we can take $(t^{(k)}, \tilde{y}^{(k)}(\cdot)) \in Y$ such that $\varphi(t^{(k)}, x^{(k)}(\cdot)) \geq \varphi(t^{(k)}, \tilde{y}^{(k)}(\cdot)) - \varepsilon_\ast \delta / 2$ and, consequently, get
                \begin{equation*}
                    \varphi(t_\ast, x_\ast(\cdot)) + \psi_k(t_\ast, x_\ast(\cdot))
                    \geq \varphi(t^{(k)}, \tilde{y}^{(k)}(\cdot)) - \varepsilon_\ast (\tau^{(k)} - t_\ast + \delta / 2)
                    + \psi_k(t^{(k)}, x^{(k)}(\cdot)).
                \end{equation*}
                Since the set $Y$ is compact, we can assume that there exists $(\tilde{t}, \tilde{y}(\cdot)) \in D$ such that
                \begin{equation*}
                    \rho_\infty \bigl( (t^{(k)}, \tilde{y}^{(k)}(\cdot)), (\tilde{t}, \tilde{y}(\cdot)) \bigr)
                    \to 0
                    \text{ as } k \to \infty.
                \end{equation*}
                Recalling that $t^{(k)} \to \bar{\tau}$ as $k \to \infty$, we obtain $\tilde{t} = \bar{\tau}$.
                Hence, and due to continuity of $\varphi$ (actually, lower semi-continuity of $\varphi$ is enough), we derive
                \begin{equation*}
                    \varphi(t_\ast, x_\ast(\cdot))
                    \geq \varphi(\bar{\tau}, \tilde{y}(\cdot)) - \varepsilon_\ast (\bar{\tau} - t_\ast + \delta / 2)
                    = \varphi(t_\ast + \delta, \tilde{y}(\cdot)) - 3 \varepsilon_\ast \delta / 2,
                \end{equation*}
                which yields a contradiction with the choice of $\delta$ and completes the proof of inequality \eqref{overline_tau_leq}.
                On the other hand, the proposed replacement of $X$ with $X_\ast$ leads to difficulties with the proof of the inclusion $(p_0, p) \in D^- \varphi(t^{(k)}, x^{(k)}(\cdot))$, where we now need to show that $(t^{(k)} + \delta_z, z_{t^{(k)} + \delta_z}(\cdot)) \in X_\ast$.
                Namely, arguing similarly to the given proof, we can conclude that, since $(t^{(k)}, x^{(k)}(\cdot)) \in X_\ast$, there exists $\omega_\ast(\cdot) \in \Omega$ such that
                \begin{equation} \label{z_first_inequality}
                    \|z(\xi) - \omega_\ast(\xi)\|
                    \leq \theta,
                    \quad \forall \xi \in [t_\ast, t^{(k)}],
                \end{equation}
                and, due to the inclusion $(\tau^{(k)}, y^{(k)}(\cdot)) \in Y$, and under an appropriate choice of the parameters, there exists $\omega^\ast(\cdot) \in \Omega$ such that
                \begin{equation} \label{z_second_inequality}
                    \|z(\xi) - \omega^\ast(\xi)\|
                    \leq \theta,
                    \quad \forall \xi \in (t^{(k)}, t^{(k)} + \delta_z].
                \end{equation}
                However, in order to get that $(t^{(k)} + \delta_z, z_{t^{(k)} + \delta_z}(\cdot)) \in X_\ast$, it is required to find a single function from $\Omega$ for which both inequalities \eqref{z_first_inequality} and \eqref{z_second_inequality} hold simultaneously.

            \smallskip

            \noindent (v)
                It can be readily seen from the given proof of the lemma that the assumption $\varphi \in \Phi_+$ can be weakened, for example, as follows (in this connection, see, e.g., \cite[Proposition 4]{Lukoyanov_Plaksin_2019_MIAN_Eng}).
                Let us suppose that, for a functional $\varphi \colon G \to \mathbb{R}$, there are numbers $h_0 \in (0, h]$, $M \in \mathbb{N}$, and $\{\vartheta^{(m)}\}_{m = 1}^M \subset [- h, T]$ such that, if we consider the corresponding metric $\rho_1^\ast \colon G \times G \to \mathbb{R}$ given by (compare with \eqref{rho_1})
                \begin{multline*}
                    \rho_1^\ast \bigl( (t, x(\cdot)), (\tau, y(\cdot)) \bigr)
                    = |t - \tau|
                    + \|x(t) - y(\tau)\|
                    + \int_{- h}^{T} \|x(\xi \wedge t) - y(\xi \wedge \tau)\| \, \rd \xi \\
                    + \max_{\xi \in [- h, (t - h_0) \wedge (\tau - h_0)]} \|x(\xi) - y(\xi)\|
                    + \sum_{\vartheta^{(m)} \in [- h, t \wedge \tau]} \|x(\vartheta^{(m)}) - y(\vartheta^{(m)})\|
                \end{multline*}
                for all $(t, x(\cdot))$, $(\tau, y(\cdot)) \in G$, then the following lower semi-continuity property holds:
                for every point $(t, x(\cdot)) \in G$ and every sequence $\{(t^{(k)}, x^{(k)}(\cdot))\}_{k = 1}^\infty \subset G$, inequality \eqref{Phi_+} is valid provided that $\rho_1^\ast ((t^{(k)}, x^{(k)}(\cdot)), (t, x(\cdot))) \to 0$ as $k \to \infty$ and there exists $\alpha > 0$ such that $\{(t^{(k)}, x^{(k)}(\cdot))\}_{k = 1}^\infty \subset G(\alpha)$.
                Then, the only change that we need to make in the proof is to choose a number $\delta > 0$ satisfying inequalities \eqref{choice_delta} and also such that $\delta \leq h_0$ and the interval $(t_\ast, t_\ast + \delta]$ does not contain any of the points $\vartheta^{(m)}$ for all $m \in \overline{1, M}$.
                Indeed, it suffices to note that, in this case, the equality below is fulfilled:
                \begin{equation*}
                    \rho_1^\ast \bigl( (t, x(\cdot)), (\tau, y(\cdot)) \bigr)
                    = \rho_1 \bigl( (t, x(\cdot)), (\tau, y(\cdot)) \bigr),
                    \quad \forall (t, x(\cdot)) \in X, \quad \forall (\tau, y(\cdot)) \in Y.
                \end{equation*}
        \end{remark}

        In conclusion of this section, let us make a comment additional to the discussion from Section \ref{subsubsection_Lipschitz}.
        Note that, similarly to \eqref{derivatives_multivalued_via_single_valued}, for any functional $\varphi \in \Phi_{\Lip}$, any point $(t, x(\cdot)) \in G_0$, and any non-empty convex compact set $L \subset \mathbb{R}^n$, we have
        \begin{equation*}
            \rd^-_0 \varphi(t, x(\cdot); L)
            = \inf_{l \in L} \partial_\ast^- \varphi(t, x(\cdot); l).
        \end{equation*}
        This equality can be proved by essentially repeating the arguments from \cite[Section 5.4]{Lukoyanov_2006_IMM_Eng}, since, in order to obtain a relation like (5.8) in that paper, we can use definition \eqref{new_lower_directional_derivative_multivalued} of the value $\rd^-_0 \varphi(t, x(\cdot); L)$.
        In this connection, see also \cite[Proposition 4]{Lukoyanov_Plaksin_2019_MIAN_Eng}.
        Hence, Lemma \ref{lemma_CL} implies
        \begin{corollary} \label{corollary_CL}
            Let $\varphi \in \Phi_+ \cap \Phi_{\Lip}$, let $(t_\ast, x_\ast(\cdot)) \in G_0$, and let $L \subset \mathbb{R}^n$ be a non-empty convex compact set.
            Suppose that
            \begin{equation*}
                \inf_{l \in L} \partial_\ast^- \varphi(t_\ast, x_\ast(\cdot); l)
                > 0.
            \end{equation*}
            Then, for every $\eta > 0$, there exist $(t, x(\cdot)) \in O_\eta(t_\ast, x_\ast(\cdot))$ and $(p_0, p) \in D^-\varphi (t, x(\cdot))$, where $D^- \varphi(t, x(\cdot))$ is calculated by \eqref{subdifferential_finite}, such that inequality \eqref{lemma_CL_main} holds.
        \end{corollary}

        Thus, comparing Corollary \ref{corollary_CL} with \cite[Lemma 4.7]{Plaksin_2020_JOTA}, we note that an important difference is that the latter result provides the existence of a point $(t, x(\cdot))$ and a pair $(p_0, p)$ satisfying inequality \eqref{lemma_CL_main} but only with a piecewise continuous function $x(\cdot)$, which is due to a different technique used in the proof.
        In turn, this circumstance leads to the need to consider the HJ equation \eqref{HJ} in the wider space of all piecewise continuous functions.

    \subsection{Proof of Theorem \ref{theorem_main}}
    \label{subsection_proof}

        The proof of the first part of Theorem \ref{theorem_main} consists of three steps, each of which is formulated as a separate lemma.

        \begin{lemma} \label{lemma_1}
            Let Assumption \ref{assumption_H1_H2} hold, and let a functional $\varphi \colon G \to \mathbb{R}$ be lower semi-continuous.
            Then, condition \eqref{upper_minimax} implies condition \eqref{upper_viscosity}.
        \end{lemma}
        \begin{proof}
            It follows from Proposition \ref{proposition_criteria_minimax} that \eqref{upper_minimax} implies \eqref{upper_minimax_derivatives_multi-valued}.
            Hence, it suffices to show that \eqref{upper_minimax_derivatives_multi-valued} implies \eqref{upper_viscosity}.
            Note that this fact was proved in \cite[Proposition 14.1]{Lukoyanov_2011_Eng}.
            However, since the proof is available in Russian only, we present it below for the reader's convenience.

            Suppose that a functional $\varphi \colon G \to \mathbb{R}$ satisfies \eqref{upper_minimax_derivatives_multi-valued}.
            Let us take $(t, x(\cdot)) \in G_0$ and $(p_0, p) \in D^- \varphi(t, x(\cdot))$ and fix $\eta > 0$.
            Due to \eqref{upper_minimax_derivatives_multi-valued} with $s \triangleq p$, we have
            \begin{equation*}
                \rd^- \varphi_p \bigl( t, x(\cdot); B_{c_H}(t, x(\cdot)) \bigr) + H(t, x(\cdot), p)
                \leq 0.
            \end{equation*}
            Then, according to \eqref{derivatives_multi-valued}, there exist $\varepsilon > 0$ and $\omega(\cdot) \in \Omega(t, x(\cdot), [B_{c_H}(t, x(\cdot))]^\varepsilon)$ such that
            \begin{equation} \label{lemma_1_proof_1}
                \partial^- \varphi_p (t, x(\cdot); \omega(\cdot)) + H(t, x(\cdot), p)
                \leq \eta.
            \end{equation}
            On the other hand, by \eqref{subdifferential_via_directional}, we obtain
            \begin{equation} \label{lemma_1_proof_2}
                p_0
                \leq \partial^- \varphi_p (t, x(\cdot); \omega(\cdot)).
            \end{equation}
            From \eqref{lemma_1_proof_1} and \eqref{lemma_1_proof_2}, we derive $p_0 + H(t, x(\cdot), p) \leq \eta$.
            Since this estimate holds for any $\eta > 0$, we conclude that the inequality in \eqref{upper_viscosity} is valid.
            The lemma is proved.
        \end{proof}

        \begin{lemma} \label{lemma_2}
            Let Assumption \ref{assumption_H1_H2} hold, and let a functional $\varphi \colon G \to \mathbb{R}$ be lower semi-continuous and satisfy the condition
            \begin{equation} \label{upper_minimax_new_derivatives_multi-valued}
                \rd^-_0 \varphi_s \bigl( t, x(\cdot); B_{c_H}(t, x(\cdot)) \bigr) + H(t, x(\cdot), s)
                \leq 0,
                \quad \forall (t, x(\cdot)) \in G_0, \quad \forall s \in \mathbb{R}^n,
            \end{equation}
            where the functional $\varphi_s$ is defined by $\varphi$ and $s$ according to \eqref{varphi_s}.
            Then, the functional $\varphi$ satisfies condition \eqref{upper_minimax}.
        \end{lemma}
        \begin{proof}
            The proof can be carried out by repeating the arguments from the second part of the proof of \cite[Theorem 8.1]{Lukoyanov_2003_1} (see also Proposition \ref{proposition_criteria_minimax} above).
            Indeed, the only difference is that the inequality in \eqref{upper_minimax_new_derivatives_multi-valued} involves the value $\rd^-_0 \varphi_s(t, x(\cdot); B_{c_H}(t, x(\cdot)))$ (see \eqref{new_lower_directional_derivative_multivalued}) instead of the directional derivative $\rd^- \varphi_s (t, x(\cdot); B_{c_H}(t, x(\cdot)))$ (see \eqref{derivatives_multi-valued}).
            Nevertheless, it follows directly from \eqref{new_lower_directional_derivative_multivalued} that condition \eqref{upper_minimax_new_derivatives_multi-valued} is sufficient in order to obtain relations (8.15) and (8.16) from \cite{Lukoyanov_2003_1}.
            The details are omitted.
        \end{proof}

        \begin{remark}
            Let us suppose that Assumption \ref{assumption_H1_H2} hold and take a lower semi-continuous functional $\varphi \colon G \to \mathbb{R}$.
            Then, by Proposition \ref{proposition_criteria_minimax}, condition \eqref{upper_minimax} implies condition \eqref{upper_minimax_derivatives_multi-valued}, which, in turn, implies condition \eqref{upper_minimax_new_derivatives_multi-valued} due to inequality \eqref{new_and_old_lower_directional_derivatives_multivalued}.
            Thus, as a corollary of Lemma \ref{lemma_2}, we obtain that all three conditions \eqref{upper_minimax}, \eqref{upper_minimax_derivatives_multi-valued}, and \eqref{upper_minimax_new_derivatives_multi-valued} are equivalent.
        \end{remark}

        \begin{lemma} \label{lemma_3}
            Let Assumption \ref{assumption_H1_H2} hold.
            Then, condition \eqref{upper_viscosity} implies condition \eqref{upper_minimax_new_derivatives_multi-valued} for every functional $\varphi \in \Phi_+$.
        \end{lemma}

        The proof of Lemma \ref{lemma_3} uses the following lower semi-continuity property of the value $\rd^-_0 \varphi(t, x(\cdot); L)$ from \eqref{new_lower_directional_derivative_multivalued} with respect to the variation of the set $L$.
        \begin{proposition} \label{proposition_lower_semicontinuity_of_directional_derivatives}
            Let $\varphi \colon G \to \mathbb{R}$, let $(t, x(\cdot)) \in G_0$, and let $L \subset \mathbb{R}^n$ be a non-empty convex compact set.
            Then, for every $A \in \mathbb{R}$ satisfying $\rd^-_0 \varphi(t, x(\cdot); L) > A$, there exists $\nu > 0$ such that, for every non-empty convex compact set $K \subset \mathbb{R}^n$ with $K \subset [L]^\nu$, the inequality below is valid:
            \begin{equation} \label{proposition_lower_semicontinuity_of_directional_derivatives_main}
                \rd^-_0 \varphi(t, x(\cdot); K)
                \geq A.
            \end{equation}
        \end{proposition}
        \begin{proof}
            Let us choose $\delta \in (0, T - t]$ from the condition
            \begin{equation} \label{proposition_choice_delta}
                A
                \leq \inf \biggl\{ \frac{\varphi(\tau, \omega_\tau(\cdot)) - \varphi(t, x(\cdot))}{\tau - t} \biggm|
                \tau \in (t, t + \delta], \, \omega(\cdot) \in \Omega(t, x(\cdot), [L]^\delta) \biggr\}
            \end{equation}
            and set $\nu \triangleq \delta / 2 > 0$.

            Now, let $K \subset \mathbb{R}^n$ be a non-empty convex compact set such that $K \subset [L]^\nu$.
            Then, for every $\eta \in (0, \delta / 2]$, we have $(t, t + \eta] \subset (t, t + \delta]$ and $\Omega(t, x(\cdot), [K]^\eta) \subset \Omega(t, x(\cdot), [L]^\delta)$.
            Hence, and due to \eqref{proposition_choice_delta}, we obtain
            \begin{equation*}
                \inf \biggl\{ \frac{\varphi(\tau, \omega_\tau(\cdot)) - \varphi(t, x(\cdot))}{\tau - t} \biggm|
                \tau \in (t, t + \eta], \, \omega(\cdot) \in \Omega(t, x(\cdot), [K]^\eta) \biggr\}
                \geq A
            \end{equation*}
            for all $\eta \in (0, \delta / 2]$, which yields the desired inequality \eqref{proposition_lower_semicontinuity_of_directional_derivatives_main}.
        \end{proof}

    \begin{proof}[Proof of Lemma \ref{lemma_3}]
        Let $\varphi \in \Phi_+$ be a functional satisfying condition \eqref{upper_viscosity}.
        Arguing by contradiction, assume that condition \eqref{upper_minimax_new_derivatives_multi-valued} does not hold, i.e., there are $(t_\ast, x_\ast(\cdot)) \in G_0$ and $s_\ast \in \mathbb{R}^n$ for which the inequality below is valid:
        \begin{equation*}
            \rd^-_0 \varphi_{s_\ast} \bigl( t_\ast, x_\ast(\cdot); B_{c_H}(t_\ast, x_\ast(\cdot)) \bigr) + H(t_\ast, x_\ast(\cdot), s_\ast)
            > 0.
        \end{equation*}
        Then, and in view of Proposition \ref{proposition_lower_semicontinuity_of_directional_derivatives}, we can find $\nu > 0$ and $\varepsilon > 0$ such that
        \begin{equation} \label{nu_epsilon}
            \rd^-_0 \varphi_{s_\ast} \bigl( t_\ast, x_\ast(\cdot); \bigl[ B_{c_H}(t_\ast, x_\ast(\cdot)) \bigr]^\nu \bigr)
            + H(t_\ast, x_\ast(\cdot), s_\ast)
            > \varepsilon.
        \end{equation}
        According to condition (i) of Assumption \ref{assumption_H1_H2} and the definition of the set $B_{c_H}(t_\ast, x_\ast(\cdot))$ (see \eqref{B}), let us choose $\eta > 0$ such that
        \begin{equation} \label{eta}
            |H(t, x(\cdot), s_\ast) - H(t_\ast, x_\ast(\cdot), s_\ast)|
            \leq \varepsilon,
            \quad B_{c_H}(t, x(\cdot))
            \subset \bigl[ B_{c_H}(t_\ast, x_\ast(\cdot)) \bigr]^\nu
        \end{equation}
        for all $(t, x(\cdot)) \in O_\eta(t_\ast, x_\ast(\cdot))$ (see \eqref{neighbourhood}).

        Let us consider the functional $\tilde{\varphi} \colon G \to \mathbb{R}$ given by
        \begin{equation} \label{widetilde_varphi}
            \tilde{\varphi}(t, x(\cdot))
            \triangleq \varphi_{s_\ast}(t, x(\cdot)) + (H(t_\ast, x_\ast(\cdot), s_\ast) - \varepsilon) (t - t_\ast),
            \quad \forall (t, x(\cdot)) \in G.
        \end{equation}
        Due to the definitions of the set $\Phi_+$ (see Section \ref{subsection_minimax}) and the metric $\rho_1$ (see \eqref{rho_1}), we obtain that the inclusion $\varphi \in \Phi_+$ yields $\varphi_{s_\ast} \in \Phi_+$ (see \eqref{varphi_s}) and $\tilde{\varphi} \in \Phi_+$.
        In addition, taking \eqref{new_lower_directional_derivative_multivalued} into account and using \eqref{nu_epsilon}, we derive
        \begin{multline*}
            \rd^-_0 \tilde{\varphi} \bigl( t_\ast, x_\ast(\cdot); \bigl[ B_{c_H}(t_\ast, x_\ast(\cdot)) \bigr]^\nu \bigr) \\
            = \rd^-_0 \varphi_{s_\ast} \bigl( t_\ast, x_\ast(\cdot); \bigl[ B_{c_H}(t_\ast, x_\ast(\cdot)) \bigr]^\nu \bigr)
            + H(t_\ast, x_\ast(\cdot), s_\ast) - \varepsilon
            > 0.
        \end{multline*}
        Hence, applying Lemma \ref{lemma_CL}, we conclude that there exist $(t, x(\cdot)) \in O_\eta(t_\ast, x_\ast(\cdot))$ and $(\tilde{p}_0, \tilde{p}) \in D^- \tilde{\varphi}(t, x(\cdot))$ such that
        \begin{equation}\label{using_lemma}
            \tilde{p}_0 + \langle \tilde{p}, l \rangle
            > 0,
            \quad \forall l \in \bigl[ B_{c_H}(t_\ast, x_\ast(\cdot)) \bigr]^\nu.
        \end{equation}

        Set $p_0 \triangleq \tilde{p}_0 - H(t_\ast, x_\ast(\cdot), s_\ast) + \varepsilon$ and $p \triangleq \tilde{p} + s_\ast$.
        Then, by \eqref{subdifferential_via_directional} and \eqref{widetilde_varphi}, we get $(p_0, p) \in D^- \varphi(t, x(\cdot))$.
        Consequently, using condition \eqref{upper_viscosity} as well as condition (ii) of Assumption \ref{assumption_H1_H2} and the choice of $\eta$ (see \eqref{eta}), we derive
        \begin{align*}
            0
            & \geq p_0 + H(t, x(\cdot), p)
            \geq \tilde{p}_0 - H(t, x(\cdot), s_\ast) + H(t, x(\cdot), \tilde{p} + s_\ast) \\
            & \geq \tilde{p}_0 + \min_{l \in B_{c_H}(t, x(\cdot))} \langle \tilde{p}, l \rangle
            \geq \tilde{p}_0 + \min_{l \in [B_{c_H}(t_\ast, x_\ast(\cdot))]^\nu} \langle \tilde{p}, l \rangle
        \end{align*}
        and come to a contradiction with \eqref{using_lemma}.
        The lemma is proved.
    \end{proof}

    Putting together Lemmas \ref{lemma_1}--\ref{lemma_3}, we obtain that, under Assumption \ref{assumption_H1_H2}, conditions \eqref{upper_minimax} and \eqref{upper_viscosity} are equivalent for every functional $\varphi \in \Phi_+$, which is actually the first part of Theorem \ref{theorem_main} in view of Definitions \ref{definition_minimax} and \ref{definition_viscosity}.

\section{Acknowledgments}

    We would like to thank Prof. Andrea Cosso for a discussion on the subject of this paper and for pointing us to the paper \cite{Li_Shi_2000}.

\appendix

\section{Proofs of Section \ref{section_Discussion}}
\label{appendix_proofs}

    \subsection{Proofs of Propositions \ref{proposition_criteria_viscosity} and \ref{proposition_test_functionals}}
    \label{appendix_proofs_1}

        For brevity, we prove both Proposition \ref{proposition_criteria_viscosity} and Proposition \ref{proposition_test_functionals} at once.

        Let $\varphi \colon G \to \mathbb{R}$ and $(t, x(\cdot)) \in G_0$.
        Let us first show that condition \eqref{upper_viscosity} implies condition \eqref{upper_viscosity_derivatives_along_extension}.
        Let us fix $s \in \mathbb{R}^n$ and denote
        \begin{equation*}
            p^\ast_0
            \triangleq \inf_{z(\cdot) \in \Lip(t,x(\cdot))} \partial^- \varphi_s(t, x(\cdot); z(\cdot)).
        \end{equation*}
        If $p^\ast_0 = - \infty$, then the inequality in \eqref{upper_viscosity_derivatives_along_extension} holds automatically.
        So, we can assume that $p^\ast_0 > - \infty$.
        Set $p \triangleq s$.
        Then, for all $p_0 \in \mathbb{R}$ with $p_0 \leq p^\ast_0$, we have $(p_0, p) \in D^- \varphi(t, x(\cdot))$ (see \eqref{subdifferential_via_directional}), and, hence, we get $p_0 \leq - H(t, x(\cdot), p)$ by \eqref{upper_viscosity}.
        Therefore, we conclude that $p^\ast_0 < + \infty$, and, choosing $p_0 \triangleq p^\ast_0$, we obtain the inequality in \eqref{upper_viscosity_derivatives_along_extension}.

        Now, let us prove that \eqref{upper_viscosity_derivatives_along_extension} implies the condition from Proposition \ref{proposition_test_functionals}.
        Let $\psi \colon G \to \mathbb{R}$ be a $ci$-smooth functional such that, for every function $z(\cdot) \in \Lip (t, x(\cdot))$, there exists $\delta_z \in (0, T - t]$ for which \eqref{upper_viscosity_test_functions_condition} holds.
        Let us put $s \triangleq \nabla \psi(t, x(\cdot))$ and fix $\varepsilon > 0$.
        Then, using \eqref{upper_viscosity_derivatives_along_extension} and taking \eqref{varphi_s} and \eqref{derivatives_along_extenstion} into account, we can choose $z(\cdot) \in \Lip(t, x(\cdot))$ satisfying the inequality
        \begin{multline*}
            \liminf_{\tau \downarrow t} \frac{\varphi(\tau, z_\tau(\cdot)) - \varphi(t, x(\cdot))
            - \langle \nabla \psi(t, x(\cdot)), z(\tau) - x(t) \rangle}{\tau - t} \\
            + H \bigl( t, x(\cdot), \nabla \psi(t, x(\cdot)) \bigr)
            \leq \varepsilon.
        \end{multline*}
        Consequently, relying on \eqref{upper_viscosity_test_functions_condition}, we derive
        \begin{multline*}
            \liminf_{\tau \downarrow t} \frac{\psi(\tau, z_\tau(\cdot)) - \psi(t, x(\cdot))
            - \langle \nabla \psi(t, x(\cdot)), z(\tau) - x(t) \rangle}{\tau - t} \\
            + H \bigl( t, x(\cdot), \nabla \psi(t, x(\cdot)) \bigr)
            \leq \varepsilon,
        \end{multline*}
        wherefrom, due to $ci$-differentiability of $\psi$ at $(t, x(\cdot))$ (see \eqref{ci_derivatives}), we get
        \begin{equation*}
            \partial_t \psi(t, x(\cdot)) + H \bigl( t, x(\cdot), \nabla \psi(t, x(\cdot)) \bigr)
            \leq \varepsilon.
        \end{equation*}
        Since this estimate is valid for any $\varepsilon > 0$, we obtain \eqref{upper_viscosity_test_functions_statement}.

        Finally, let us assume that the condition from Proposition \ref{proposition_test_functionals} holds and verify \eqref{upper_viscosity}.
        Take $(p_0, p) \in D^- \varphi(t, x(\cdot))$, fix $\varepsilon > 0$, and consider the functional $\psi \colon G \to \mathbb{R}$ given by
        \begin{equation} \label{psi_test}
            \psi(\tau, y(\cdot))
            \triangleq (p_0 - \varepsilon) (\tau - t) + \langle p, y(\tau) - x(t) \rangle,
            \quad \forall (\tau, y(\cdot)) \in G.
        \end{equation}
        Note that $\psi$ is $ci$-smooth, $\psi(t, x(\cdot)) = 0$, and, in addition, $\partial_t \psi(t, x(\cdot)) = p_0 - \varepsilon$ and $\nabla \psi (t, x(\cdot)) = p$.
        It follows from \eqref{subdifferential} that, for every function $z(\cdot) \in \Lip(t, x(\cdot))$, we can choose $\delta_z \in (0, T - t]$ such that
        \begin{equation*}
            \varphi(\tau, z_\tau(\cdot)) - \varphi(t, x(\cdot)) - \langle p, z(\tau) - x(t) \rangle
            \geq (p_0 - \varepsilon) (\tau - t),
            \quad \forall \tau \in [t, t + \delta_z],
        \end{equation*}
        which yields \eqref{upper_viscosity_test_functions_condition} with the functional $\psi$ from \eqref{psi_test}.
        Hence, based on the assumption made (see \eqref{upper_viscosity_test_functions_statement}), we conclude that $p_0 - \varepsilon + H(t, x(\cdot), p) \leq 0$.
        Since this estimate holds for any $\varepsilon > 0$, we get the inequality in \eqref{upper_viscosity}.

        Thus, Propositions \ref{proposition_criteria_viscosity} and \ref{proposition_test_functionals} are proved.

    \subsection{Proof of Proposition \ref{proposition_finite_derivatives}}
    \label{appendix_proofs_2}
        According to \eqref{derivatives_along_extenstion}, we can find a sequence $\{\tau^{(k)}\}_{k = 1}^\infty \subset (t, T]$ for which, as $k \to \infty$,
        \begin{equation}\label{proposition_finite_derivatives_2}
            \tau^{(k)}
            \to t,
            \quad \frac{\varphi(\tau^{(k)}, z_{\tau^{(k)}}(\cdot)) - \varphi(t, x(\cdot))}{\tau^{(k)} - t}
            \to \partial^- \varphi(t, x(\cdot); z(\cdot)).
        \end{equation}
        For every $k \in \mathbb{N}$, let us denote $l_k \triangleq (z(\tau^{(k)}) - x(t)) / (\tau^{(k)} - t)$.
        Since $z(\cdot) \in \Lip(t, x(\cdot))$, let us choose $\lambda_z > 0$ from the condition $\|z(\tau) - x(t)\| \leq \lambda_z (\tau - t)$ for all $\tau \in [t, T]$.
        In particular, we have $\|l_k\| \leq \lambda_z$ for all $k \in \mathbb{N}$, and, therefore, we can assume that there exists $l \in \mathbb{R}^n$ such that $\|l\| \leq \lambda_z$ and $l_k \to l$ as $k \to \infty$.
        Consider the corresponding functions $z^{[l_k]}(\cdot) \in \Lip(t, x(\cdot))$ for all $k \in \mathbb{N}$ and $z^{[l]}(\cdot) \in \Lip(t, x(\cdot))$, defined by \eqref{z^l}.
        Now, let us take $\alpha \triangleq \|x(\cdot)\|_{[-h, t]} + \lambda_z (T - t) > 0$ and, using the assumption $\varphi \in \Phi_{\Lip}$, choose $\lambda_\varphi > 0$ for which condition \eqref{Phi_Lip} is valid.
        Then, for every $k \in \mathbb{N}$, noting that $(\tau^{(k)}, z_{\tau^{(k)}}(\cdot))$, $(\tau^{(k)}, z^{[l]}_{\tau^{(k)}}(\cdot)) \in G(\alpha)$, we derive
        \begin{align*}
            & |\varphi(\tau^{(k)}, z_{\tau^{(k)}}(\cdot)) - \varphi(\tau^{(k)}, z^{[l]}_{\tau^{(k)}}(\cdot))| \\
            & \quad \leq \lambda_\varphi \biggl( (1 + T - \tau^{(k)}) \|z(\tau^{(k)}) - z^{[l]}(\tau^{(k)})\| + \int_{t}^{\tau^{(k)}} \|z(\xi) - z^{[l]}(\xi)\| \, \rd \xi \biggr) \\
            & \quad \leq \lambda_\varphi \biggl( (1 + T - t) (\tau^{(k)} - t) \|l_k - l\| + 2 \lambda_z \int_{t}^{\tau^{(k)}} (\xi - t) \, \rd \xi \biggr) \\
            & \quad = \lambda_\varphi ((1 + T - t) \|l_k - l\| + \lambda_z (\tau^{(k)} - t)) (\tau^{(k)} - t)
        \end{align*}
        and, consequently,
        \begin{multline*}
            \frac{\varphi(\tau^{(k)}, z^{[l]}_{\tau^{(k)}}(\cdot)) - \varphi(t, x(\cdot))}{\tau^{(k)} - t} \\
            \leq \frac{\varphi(\tau^{(k)}, z_{\tau^{(k)}}(\cdot)) - \varphi(t, x(\cdot))}{\tau^{(k)} - t}
            + \lambda_\varphi ((1 + T - t) \|l_k - l\| + \lambda_z (\tau^{(k)} - t)),
        \end{multline*}
        which, in accordance with \eqref{proposition_finite_derivatives_2}, implies \eqref{proposition_finite_derivatives_0}.
        The proof is complete.

    \subsection{Proofs of Propositions \ref{proposition_value_Phi} and \ref{proposition_existence_Lip}}
    \label{appendix_proofs_3}

        Let us first prove
        \begin{propositionA} \label{value_x_continuity}
            Let Assumption \ref{assumption_optimal_control} hold.
            Then, for any $\alpha > 0$ and any $\varepsilon > 0$, there exists $\delta > 0$ such that $|\varphi_0(t, x(\cdot)) - \varphi_0(t, y(\cdot))| \leq \varepsilon$ for all $(t, x(\cdot))$, $(t, y(\cdot)) \in G(\alpha)$ with
            \begin{equation} \label{value_x_continuity_condition}
                \rho_1 \bigl( (t, x(\cdot)), (t, y(\cdot)) \bigr)
                \leq \delta.
            \end{equation}
        \end{propositionA}
        \begin{proof}
            Let $\alpha > 0$ be fixed.
            Due to conditions (i)--(iii) of Assumption \ref{assumption_optimal_control}, the following two statements can be verified by the scheme from, e.g., \cite[Propositions 3.1 and 3.2]{Plaksin_2020_JOTA} (see also \cite[Lemmas 4.1 and 4.2]{Plaksin_2021_SIAM}):

            \smallskip

            \noindent (i)
                There exists $\alpha_\ast \geq \alpha$ such that, for every $(t, x(\cdot)) \in G(\alpha)$ and every $u(\cdot) \in \mathcal{U}(t)$, the corresponding motion $z(\cdot) \triangleq z(\cdot; t, x(\cdot), u(\cdot))$ of system \eqref{system}, \eqref{initial_condition} satisfies the inclusion $(\tau, z_\tau(\cdot)) \in G(\alpha_\ast)$ for all $\tau \in [t, T]$.

            \smallskip

            \noindent (ii)
                There exists $\lambda_\ast > 0$ such that, for any $(t, x(\cdot))$, $(t, y(\cdot)) \in G(\alpha)$ and any $u(\cdot) \in \mathcal{U}(t)$, the corresponding system motions $z(\cdot) \triangleq z(\cdot; t, x(\cdot), u(\cdot))$ and $w(\cdot) \triangleq z(\cdot; t, y(\cdot), u(\cdot))$ satisfy the estimate
                \begin{multline*}
                    \rho_1 \bigl( (T, z(\cdot)), (T, w(\cdot)) \bigr) +
                    \biggl| \int_t^T \bigl( \chi(\xi, z_\xi(\cdot), u(\xi)) - \chi(\xi, w_\xi(\cdot), u(\xi)) \bigr) \, \rd \xi \biggr| \\
                    \leq \lambda_\ast \rho_1 \bigl( (t, x(\cdot)), (t, y(\cdot)) \bigr).
                \end{multline*}

            Let $\varepsilon > 0$ be fixed.
            Using condition (iv) of Assumption \ref{assumption_optimal_control}, let us take $\delta_\sigma > 0$ such that the inequality $|\sigma(z(\cdot)) - \sigma(w(\cdot))| \leq \varepsilon / 3$ is valid for all $(T, z(\cdot))$, $(T, w(\cdot)) \in G(\alpha_\ast)$ with $\rho_1 ((T, z(\cdot)), (T, w(\cdot))) \leq \delta_\sigma$.
            Let us choose $\delta > 0$ from the conditions
            \begin{equation} \label{value_x_continuity:delta}
                \delta
                \leq \delta_\sigma / \lambda_\ast,
                \quad \delta
                \leq \varepsilon / (3 \lambda_\ast).
            \end{equation}

            Now, let $(t, x(\cdot))$, $(t, y(\cdot)) \in G(\alpha)$ and let \eqref{value_x_continuity_condition} be fulfilled.
            Let us show that
            \begin{equation} \label{value_x_continuity:statement}
                \varphi_0(t, y(\cdot)) - \varphi_0(t, x(\cdot))
                \leq \varepsilon.
            \end{equation}
            By definition \eqref{value} of the value functional $\varphi_0$, there exists $u(\cdot) \in \mathcal{U}(t)$ such that, for the system motion $z(\cdot) \triangleq z(\cdot; t, x(\cdot), u(\cdot))$, it holds that
            \begin{equation} \label{value_x_continuity:y}
                \varphi_0(t, x(\cdot))
                \geq \sigma(z(\cdot)) + \int_t^T \chi(\xi, z_\xi(\cdot), u(\xi)) \, \rd \xi - \varepsilon / 3.
            \end{equation}
            In addition, for the system motion $w(\cdot) \triangleq z(\cdot; t, y(\cdot), u(\cdot))$, we have
            \begin{equation} \label{value_x_continuity:y'}
                \varphi_0(t, y(\cdot))
                \leq \sigma(w(\cdot)) + \int_t^T \chi(\xi, w_\xi(\cdot), u(\xi)) \, \rd \xi.
            \end{equation}
            According to the choice of $\lambda_\ast$ and inequalities \eqref{value_x_continuity:delta}, we derive
            \begin{equation} \label{value_x_continuity:first}
                \rho_1 \bigl( (T, z(\cdot)), (T, w(\cdot)) \bigr)
                \leq \lambda_\ast \delta
                \leq \delta_\sigma
            \end{equation}
            and
            \begin{equation} \label{value_x_continuity:differences_chi}
                \biggl| \int_t^T \bigl( \chi(\xi, z_\xi(\cdot), u(\xi)) - \chi(\xi, w_\xi(\cdot), u(\xi)) \bigr) \, \rd \xi \biggr|
                \leq \lambda_\ast \delta
                \leq \varepsilon / 3.
            \end{equation}
            Since $(T, z(\cdot))$, $(T, w(\cdot)) \in G(\alpha_\ast)$, it follows from \eqref{value_x_continuity:first} and the choice of $\delta_\sigma$ that
            \begin{equation} \label{value_x_continuity:differences_sigma}
                |\sigma(z(\cdot)) - \sigma(w(\cdot))|
                \leq \varepsilon / 3.
            \end{equation}
            Thus, putting together \eqref{value_x_continuity:y}, \eqref{value_x_continuity:y'}, \eqref{value_x_continuity:differences_chi}, and \eqref{value_x_continuity:differences_sigma}, we obtain \eqref{value_x_continuity:statement}.
            The inequality $\varphi_0(t, x(\cdot)) - \varphi_0(t, y(\cdot)) \leq \varepsilon$ can be verified in a similar way.
        \end{proof}

        \begin{proof}[Proof of Proposition \ref{proposition_value_Phi}]
            Let $\alpha > 0$, and let a sequence $\{(t^{(k)}, x^{(k)}(\cdot))\}_{k = 1}^\infty \subset G(\alpha)$ and a point $(t, x(\cdot)) \in G$ be such that $\rho_1 ((t, x(\cdot)), (t^{(k)}, x^{(k)}(\cdot))) \to 0$ as $k \to \infty$.
            Note that $(t, x(\cdot)) \in G(\alpha)$ and $\rho_\infty ((t, x(\cdot)), (t^{(k)}, x_{t^{(k)}}(\cdot \wedge t^{(k)}))) \to 0$ as $k \to \infty$.
            Then, taking \eqref{rho_1_and_rho_infty} into account and applying Proposition A.1, we obtain
            \begin{equation}\label{value_is_viscosity_solution:2}
                |\varphi_0(t^{(k)}, x^{(k)}(\cdot)) - \varphi_0(t^{(k)}, x_{t^{(k)}}(\cdot \wedge t^{(k)}))|
                \to 0
                \text{ as } k \to \infty.
            \end{equation}
            Moreover, since the value functional $\varphi_0$ is continuous, we have
            \begin{equation}\label{value_is_viscosity_solution:1}
                \lim_{k \to \infty} \varphi_0(t^{(k)}, x_{t^{(k)}}(\cdot \wedge t^{(k)}))
                = \varphi_0(t, x(\cdot)).
            \end{equation}
            From \eqref{value_is_viscosity_solution:2} and \eqref{value_is_viscosity_solution:1}, we derive
            \begin{equation*}
                \lim_{k \to \infty} \varphi_0(t^{(k)}, x^{(k)}(\cdot))
                = \varphi_0(t, x(\cdot)),
            \end{equation*}
            which yields both \eqref{Phi_+} and \eqref{Phi_-}, and, thus, proves the inclusion $\varphi_0 \in \Phi_+ \cap \Phi_-$.
        \end{proof}

        Proposition \ref{proposition_existence_Lip} can be proved in the same way as Proposition A.1, but relying on Assumption \ref{assumption_sigma_Lipschitz_continuous} instead of condition (iv) of Assumption \ref{assumption_optimal_control}.

\biboptions{sort&compress}

\end{document}